\documentclass{amsart}
\usepackage[utf8]{inputenc}
\usepackage{amsmath}
\usepackage{amssymb}
\usepackage{graphicx}
\usepackage{amsfonts}
\usepackage[utf8]{inputenc}
\usepackage{hyperref}
\usepackage{amsmath}
\usepackage{amsmath,amsthm,amssymb,amsfonts}
\usepackage[utf8]{inputenc}
\usepackage{enumerate}
\usepackage{graphicx}
\usepackage{enumitem}
\usepackage[numbers]{natbib}

\theoremstyle{plain}
\newtheorem{theorem}{Theorem}[section]
\newtheorem{lemma}[theorem]{Lemma}
\newtheorem{claim}[theorem]{Claim}
\newtheorem{corollary}[theorem]{Corollary}
\newtheorem{proposition}[theorem]{Proposition}

\newtheorem{question}[theorem]{Question}

\newtheorem{problem}[theorem]{Problem}
\newtheorem{fact}[theorem]{Fact}

\theoremstyle{definition}\newtheorem{definition}[theorem]{Definition}

\theoremstyle{definition}
\theoremstyle{definition}\newtheorem{remark}[theorem]{Remark}

\numberwithin{equation}{section}

\newcommand{\N}{\mathbb{N}}

\newcommand{\om}{\mathbb{N}}
\newcommand{\omm}{[\mathbb{N}]^\mathbb{N}}
\newcommand{\oom}{\mathbb{N}^\mathbb{N}}

\newcommand{\bd}{\begin{definition}}
	\newcommand{\ed}{\end{definition}}

\DeclareMathOperator{\proj}{proj}
\DeclareMathOperator{\ran}{ran}

\newcommand{\mc}{\mathcal}
\newcommand{\mb}{\mathbf}
\newcommand{\bs}{\mathbf{\Sigma}^1_1}

\newcommand{\bbg}{\mathbf{\Gamma}}
\newcommand{\bbo}{\mathbf{\Delta}^1_1}
\newcommand{\bp}{\mathbf{\Pi}^1_1}
\newcommand{\ls}{\Sigma^1_1}
\newcommand{\lbo}{\Delta^1_1}
\newcommand{\lp}{\Pi^1_1}

\newcommand{\id}{id}
\newcommand{\concatt}{%
	\mathbin{\raisebox{1ex}{\scalebox{.7}{$\frown$}}}%
}
\newcommand{\ovl}{\overline}
\newcommand{\pred}{pred}
\begin{document}
	
	\title[A complexity problem for
	Borel graphs]{A complexity problem for
		Borel graphs}

	\author[Stevo Todor\v{c}evi\'c]{Stevo Todor\v{c}evi\'c}
	
	\author[Zolt\'an Vidny\'anszky]{Zolt\'an Vidny\'anszky}
	\thanks{The first author's research was partially supported by grants of NSERC (455916) and CNRS (IMJ-PRG UMR7586). The second author was partially supported by the
		National Research, Development and Innovation Office
		-- NKFIH, grants no.~113047, no. ~129211 and FWF Grants P29999 and M2779. }

	\insert\footins{\footnotesize{MSC codes: Primary 03E15; Secondary 28A05, 05C15}}
	\insert\footins{\footnotesize{Key Words: $\Sigma^1_2$-complete, chromatic number, Borel chromatic number, Borel graph, Hedetniemi's conjecture, antibasis}}

	\begin{abstract} We show that there is no simple (e.g. finite or countable) basis for Borel graphs with infinite Borel chromatic number. In fact, it is proved that the closed subgraphs of the shift graph on $\omm$ with finite (or, equivalently, $\leq 3$) Borel chromatic number form a $\mathbf{\Sigma}^1_2$-complete set. This answers a question of Kechris and Marks and strengthens several earlier results.
		
	\end{abstract}

	\maketitle
	\section{Introduction}

	A \emph{Borel graph} $\mc{G}$ is a pair $(X,E)$, where $X$ is a Polish space and $E \subset X^2\setminus \{(x,x):x \in X\}$ is a symmetric Borel set. The elements of $X$ are called \emph{vertices}, while the pairs in $E$ are called \emph{edges}. 
	
	The study of Borel graphs and generalizations of classical graph theoretic notions to this context is a flourishing field. One of the most natural such notions is the so called Borel chromatic number introduced in \cite{kechris1999borel}. For $n \in \{1,2,\dots,\aleph_0\}$ a Borel graph $\mc{G}=(X,E)$ is said to have a \textit{Borel chromatic number $n$}, in notation $\chi_B(\mc{G})=n$, if $n$ is minimal such that there exists a \textit{Borel $n$-coloring of $\mc{G}$}, that is, there exist a Polish space $Y$ and a Borel map $c:X \to Y$ so that $xEy$ implies $c(x) \not =c(y)$ and the size of the image of $c$ is $n$. If $\chi_B(\mc{G}) \not \leq n$ for every $n \in \{1,2,\dots,\aleph_0\}$ then we say that $\mc{G}$ has uncountable Borel chromatic number.
	
	How can the Borel chromatic number of a graph be decided? An obvious lower bound can be given if it contains a copy of a graph with a known Borel chromatic number. More precisely, we say that $\mc{H}$ is \textit{Borel below} $\mc{G}$, or $\mc{H} \leq_B \mc{G}$, if there exists a Borel map $f$ from the vertex set of $\mc{H}$ to the vertex set of $\mc{G}$ that takes edges to edges. If moreover, $f$ is a bijection and takes non-edges to non-edges as well, then $\mc{H}$ is said to be \textit{Borel isomorphic to} $\mc{G}$. It is clear that $\mc{H} \leq_B \mc{G}$ implies $\chi_B(\mc{H}) \leq \chi_B(\mc{G})$. 
	
	Kechris, Solecki and Todor\v{c}evi\'c \cite{kechris1999borel} characterized the graphs that have an uncountable Borel chromatic number proving the so called $\mathbb{G}_0$ dichotomy, that is, showing that there exists a Borel graph, called $\mathbb{G}_0$, so that a Borel graph $\mc{G}$ has uncountable Borel chromatic number if and only if $\mathbb{G}_0 \leq_B \mc{G}$. This dichotomy has plenty of applications, for instance, it implies a large collection of dichotomies in descriptive set theory, see, e.g., \cite{miller2012graph}.
	
	Thus, it is very natural to ask, whether there exists an analogue of this dichotomy for graphs with infinite Borel chromatic number. The simplest nontrivial examples of graphs with countably infinite Borel chromatic number are the graphs defined by functions: let $f:X \to X$ be a Borel map, define $\mc{G}_f=(X,E_f)$ by $xE_fy \iff x \not = y$ and $(f(x)=y$ or $f(y)=x)$. It is not hard to see that for any $f$ we have $\chi_B(\mc{G}_f) \leq \aleph_0$, see \cite[Proposition 4.5]{kechris1999borel}. 
	
	One of the most interesting instances of graphs of this sort is the \emph{shift graph}, $\mc{G}_\mc{S}$, on $\omm$ (the collection of infinite subsets of the natural numbers with the topology inherited from $\oom$). Define the \emph{shift map} by  $\mc{S}(x)=x\setminus \{\min x\}$ and let $\mc{G}_\mc{S}=(\omm, E_\mc{S})$. As mentioned above $\chi_B(\mc{G}_\mc{S})  \leq \aleph_0$.  Typically, giving a lower estimate on a Borel graph's chromatic number goes through an argument that uses the Baire category theorem (e. g., the graph $\mathbb{G}_0$), measure and ergodic theory (see \cite{conleymeas}) or the Borel determinacy theorem (see \cite{marksdet}). In our case, the lower estimate uses an infinite dimensional analogue of the Ramsey theorem, namely the Galvin-Prikry theorem. It states that for any finite cover of $\omm$ by Borel sets $B_0,\dots,B_n$ there exists an $i \leq n$ and an $x \in \omm$ so that $[x]^\om \subset B_i$, in other words,  all infinite subsets of $x$ are contained in $B_i$. This of course implies $\chi_B(\mc{G}_\mc{S})=\aleph_0$. The Galvin-Prikry theorem (in a sense that can be made precise, see \cite{simpson}) is somewhat weaker than the Borel determinacy theorem and thus the proof of $\chi_B(\mathcal{G}_\mc{S})=\aleph_0$ potentially can be considered as an example of a fourth kind. 
	
	Since $\mc{G}_\mc{S}$ is in some sense rather small (e.g. it is locally finite) but still has  infinite Borel chromatic number and certain universality properties, one might wonder whether a graph $\mc{G}$ has infinite Borel chromatic number if and only if  $\mc{G}_\mc{S} \leq_B \mc{G}$. Unfortunately, it is not hard to see that that the answer to this question is negative: the direct sum for $n \in \mathbb{N}$ of the complete finite graphs on $n$ vertices is a counterexample.  Another, much more general example to the failure of this type of basis results has been given by Conley and Miller \cite{conley2014antibasis}.
	
	After this, there are several natural ways to proceed. 
	
	Firstly, we could restrict ourselves to a smaller class of graphs, and hope for a basis result in that class. For instance, Kechris, Solecki and Todor\v{c}evi\'c asked whether being Borel above $\mc{G}_\mc{S}$ characterizes the graphs with infinite Borel chromatic number of the form $\mc{G}_f$?  Or, it is also natural to consider the structure of the \emph{Borel/closed subgraphs} of the shift graph: for a Borel graph $\mc{G}=(X,E)$ and $B \subset X$ let us denote by $\mc{G}|_B$ the graph $(X,E \cap B^2)$. 
	
	\begin{question}
		\label{q:basic}
		Let $C \subset \omm$ be a closed set. 
		\begin{enumerate}
			\item \label{q:cube} Is it true that  $\chi_B(\mc{G}_\mc{S}|_C)=\aleph_0$ if and only if there exists an $x \in \omm$ with $[x]^\om \subset C$?
			
			\item \label{q:yann} Is it true that  $\chi_B(\mc{G}_\mc{S}|_C)=\aleph_0$ if and only if $\mc{G}_\mc{S} \leq_B \mc{G}_\mc{S}|_C$?
			
		\end{enumerate}
	\end{question}
	
	A negative answer to question \eqref{q:cube} has been given by Di Prisco and Todor\v{c}evi\'c \cite{di2015basis} (see also \cite[4(E)]{kechris2015descriptive} for a simple counterexample). Moreover, it has been shown recently by Pequignot \cite{pequignot2017finite} that \eqref{q:yann} is false as well. 
	
	Secondly, one could hope for a different graph, or a countable basis instead of a one element basis:
	
	\begin{question}
		(Kechris, Marks \cite[Problem 4.23]{kechris2015descriptive}) 
		\label{q:kech} Is there a sequence $(\mc{G}_n)_{n \in \N}$	of Borel graphs with
		$\chi_B(\mc{G}_n)<\aleph_0$ and $\chi_B(\mc{G}_n)$ unbounded such that for every Borel graph $\mathcal{H}$ with infinite Borel
		chromatic number and for every $n$ we have that
		$\mc{G}_n \leq_B \mc{H}$?
	\end{question}

	It follows from our results that the answer to all of these questions is negative. Roughly speaking, positive basis results typically imply that the complexity of the collection of the Borel graphs with infinite Borel chromatic number (with an appropriate coding) is low and we will show that this is not the case, even for the closed subgraphs of $\mc{G}_\mc{S}$. Note that for such graphs having infinite Borel chromatic number is equivalent to having Borel chromatic number $\geq 4$, see \cite[Theorem 5.1]{kechris1999borel}. In \cite{benen}, Carroy, Miller, Schrittesser and the second author characterized the graphs with Borel chromatic number $\geq 3$ similarly to the $\mathbb{G}_0$ dichotomy: it is shown that there exists a Borel graph, called $\mathbb{G}_{odd}$, having the property that for every Borel graph $\mc{G}$ we have $\chi_B(\mc{G}) \geq 3$ if and only if $\mathbb{G}_{odd} \leq_B \mc{G}$. Hence we obtain a complete description of the characterizability of the Borel chromatic numbers of graphs in terms of simple bases. These results reinforce the experience from the classical case of finite graphs, namely, that it is strictly more complicated to decide whether a graph has chromatic number $\geq n$ than to check whether a given coloring is correct for every $n$, except if $n \leq 3$.  
	
	Now we formulate the precise statement of our results. A family $\mc{A}$ of Borel graphs is called \textit{$\mathbf{\Sigma}^1_1$-parametrizable} if there exist Polish spaces $X, Y$ and a $\mathbf{\Sigma}^1_1$ set $E \subset X \times Y^2$ so that for any $\mc{G} \in \mc{A}$ there exists an $x \in X$ with $\mc{G}$ being Borel isomorphic to $(Y,\{(y,z):(x,y,z)\in E\})$ and the set $\{x:(Y,\{(y,z):(x,y,z)\in E\}) \text{ is  Borel isomorphic to some graph in }\mc{A}\}$ is $\mathbf{\Sigma}^1_1$. 
	Recall that a subset $A$ of a Polish space $X$ is \textit{$\mb{\Sigma}^1_2$-hard}, if for any Polish space $X'$ and $A'$ $\mb{\Sigma}^1_2$ subset of $X'$ there exists a Borel map, called a \textit{reduction,} $f:X' \to X$ with $x' \in A' \iff f(x') \in A$ for every $x'\in X'$. A $\mb{\Sigma}^1_2$-hard set that is $\mb{\Sigma}^1_2$ is called \textit{$\mb{\Sigma}^1_2$-complete}.
	Now we are ready to state our main result.
	
	\begin{theorem}
		\label{t:mainintro}
		The collection of closed sets $C \subset \omm$ so that $\chi_B(\mc{G}_\mc{S}|_C)<\aleph_0$ (or, equivalently, $\chi_B(\mc{G}_\mc{S}|_C)\leq 3$) is $\mathbf{\Sigma}^1_2$-complete. Consequently, there exists no sequence of $\mathbf{\Sigma}^1_1$-parametrizable collections of graphs $(\mc{A}_n)_{n \in \mathbb{N}}$ so that for every $C \subset \omm$ closed set $\chi_B(\mc{G}_\mc{S}|_C)\geq \aleph_0$ if and only if there exist an increasing sequence $(n_i)_{i\in \mathbb{N}}$ and $A_{n_i} \in \mc{A}_{n_i}$ with $A_{n_i} \leq_B \mc{G}_\mc{S}|_C$. In particular, there is no one element basis, or countable basis in the sense of Question \ref{q:kech}.

	\end{theorem}
	
	Let us point out that Pequignot also used complexity to answer \eqref{q:yann} of Question \ref{q:basic}. His argument is built on a result of Marcone \cite{marcone}, who proved that the set of quasi-orders that are better quasi-orders (bqos) is $\mathbf{\Pi}^1_2$-complete, that is, its complement is $\mathbf{\Sigma}^1_2$-complete (bqos were defined by Nash-Williams \cite{nashwilliams}, they form a particularly well behaving class of quasi-orders, see also \cite[Chapter 9]{mansfield} and \cite{Peq}.) Pequignot's proof proceeds by showing that there is a reduction from bqos to the family of closed sets $C \subset \omm$ for which $\mc{G}_\mc{S} \not \leq_B \mc{G}_\mc{S}|_C$ holds. This implies that the collection $\{C\subset \omm:C \text{ is closed, } \mc{G}_\mc{S} \leq_B \mc{G}_\mc{S}|_C\}$ is $\mathbf{\Sigma}^1_2$-complete. As the set $\{C\subset \omm:C \text{ is closed, } \chi_B(\mc{G}_\mc{S}|_C)=\aleph_0\}$ is $\mathbf{\Pi}^1_2$  these sets must be distinct, i. e., the answer to \eqref{q:yann} of Question \ref{q:basic} is negative.

	On the positive side, we show that closed subgraphs of the shift with infinite Borel chromatic number form a basis for Borel subgraphs with infinite Borel chromatic number. In a sense, this answers \cite[Problem 4.22]{kechris2015descriptive}.
	
	\begin{theorem}
		\label{t:mainintro2} Let $B \subset \omm$ be a Borel set.  There exists a closed set $C \subset \omm$ and a continuous, shift-invariant injection $\Psi:C \to B$, so that $\Psi^{-1}$ is also shift-invariant and $\chi_B(\mc{G}_\mc{S}|_{C})=\chi_B(\mc{G}_\mc{S}|_{B})$. If $B$ is closed under the shift map, then $\Psi$ can be taken to be a bijection.
	\end{theorem}
	
	In order to prove the complexity result we isolate a general theorem about the complexity of certain families of sets. Suppose that we are given a family $\mc{F}^\Phi$ of subsets of a Polish space $X$ coming from a map $\Phi$ that assigns to each set in $\mc{F}^\Phi$ the set of the witnesses of being in $\mc{F}^\Phi$ (e.g. the codes of the possible finite colorings). Suppose moreover that we put sets from $\mc{F}^\Phi$ ``next to each other" i. e., consider a Borel set $B \subset \N^\N \times X$  and we are interested whether the sections of $B$ are in $\mc{F}^\Phi$ uniformly, that is, whether we can find witnesses of being in $\mc{F}^\Phi$ in a Borel way (we will see later that in the case of graph colorings this is precisely equivalent to the existence of a finite coloring of the graph obtained by putting the graphs $B_s$ ``next to each other"). How hard is it to decide the existence of such a uniform selection? Our general theorem says that if the family $\mc{F}^\Phi$ is complicated enough then it is $\mathbf{\Sigma}^1_2$-hard. (In our case, this will follow from the observations that non-dominating sets are complicated and that $\chi_B(G_\mc{S}|_B) \leq 3$ holds if $B$ is a non-dominating Borel set).

	Now we make the above idea precise. Let $X,Y$ be uncountable Polish spaces, $\bbg$ be a class of Borel sets and $\Phi:\bbg(X) \to \bp(Y)$ be a map.  Define $\mc{F}^{\Phi}\subset \bbg(X)$ by $A \in \mc{F}^{\Phi} \iff \Phi(A) \not = \emptyset$ 
	and let the \emph{uniform family, $\mc{U}^{\Phi}$,} be defined as follows: for $B \in \bbg(\om^\om \times X)$ let
	\[\bar{\Phi}(B)=\{(s,y) \in  \om^\om \times Y:y \in \Phi(B_s)\},\]
	and 
	\[B \in \mc{U}^{\Phi} \iff \bar{\Phi}(B) \text{ has a full Borel uniformization}\]
	
	(that is, it contains the graph of a Borel function $\oom \to Y$).
	
	A family $\mc{F}$ of subsets of a Polish space $X$ is said to be \emph{$\bs$ (resp. $\mathbf{\Sigma}^1_2$)-hard on $\bbg$}, if there exists a set $B \in \bbg(\om^\om \times X)$ so that the set $\{s \in \om^\om:B_s \in \mc{F}\}$ is $\bs$ (resp. $\mathbf{\Sigma}^1_2$)-hard. One would be tempted to think that the fact that $\mc{F}^\Phi$ is $\bs$-hard on $\bbg$ is sufficient for proving the $\mathbf{\Sigma}^1_2$-hard on $\mathbf{\Gamma}$-ness of the family $\mc{U}^{\Phi}$. Unfortunately, this is not the case (at least under the axiom of constructibility), see Remark \ref{r:niceness}.  On the positive side, the typical way of proving that $\mc{F}^\Phi$ is  $\bs$-hard is to start with a given  $A \in \bs(\om^\om)$ and find a closed set $D \subset \om^\om \times \oom$ with $A=\proj_0(D)$. Now, one constructs a set $B \subset \om^\om \times X$ so that $B_s \in \mc{F}^\Phi$ iff $s \in A$ and this is witnessed by every element of the set $D_s$, i. e., $D_s \subset \Phi(B_s)$. The following definition encompasses this situation.

	\begin{definition}
		\label{d:nicely}
		The  family $\mc{F}^{\Phi}$ is said to be \emph{nicely $\bs$-hard on $\bbg$} if for every $A \in \bs(\om^\om)$  there exist sets $B \in \bbg(\om^\om \times X)$ and $D \in \bs(\om^\om \times Y)$ so that $D \subset \bar{\Phi}(B)$ and for all $s \in \om^\om$ we have 
		\[  s \in A \iff D_s \not = \emptyset \iff \Phi(B_s)\not= \emptyset \ ( \ \iff B_s \in \mc{F}^\Phi).\]
		
	\end{definition}	
	We are ready to state our theorem. 
	
	\begin{theorem}
		
		\label{t:main2}
		Let $X,Y$ be uncountable Polish spaces, $\bbg$ be a class of subsets of Polish spaces which is closed under continuous preimages, finite unions and intersections and $\mathbf{\Pi}^0_1 \cup \mathbf{\Sigma}^0_1 \subset \bbg$.   Suppose that $\Phi:\bbg(X) \to \mathbf{\Pi}^1_1(Y)$ is $\bp$ on $\bbg$ (see Definition \ref{d:ond}) and that $\mc{F}^\Phi$ is nicely $\bs$-hard on $\bbg$. 
		Then the family $\mc{U}^\Phi$ is $\mathbf{\Sigma}^1_2$-hard on $\bbg$.
	\end{theorem}

	The paper is organized as follows. First we start with summarizing the most important facts and notations used in the proofs. Then in Section \ref{s:general} we prove Theorem \ref{t:main2}. In Section \ref{s:appl1} we apply this result to calculate the complexity of the collection of closed subgraphs of the shift graph and also show Theorem \ref{t:mainintro2}. Finally, in the last section we discuss a counterexample to the $\Delta^1_1$ version of Hedetniemi's conjecture and finish with a couple of open problems.
	
	\textbf{Acknowledgements.} We would like to thank Benjamin Miller for the inspiring discussions, questions and for pointing out a way to prove the analytic-hardness of non-dominating sets using only classical tools (see Lemma \ref{l:forH}). We are also very grateful to Slawomir Solecki, Alexander Kechris,  M\'arton Elekes and Jan Greb\'ik for their help, valuable comments and suggestions.

	\section{Preliminaries and notations}
	\label{s:prel}
	For the collection of finite, (resp. infinite) sequences of elements of a set $A$ the notations $A^{<\om},$ (resp. $A^{\om}$) will be used, while the family of countably infinite subsets of $A$ is denoted by $[A]^{\om}$. If $x \in A^{\om}$ and $n \in \om$ then $x|_n$ will stand for the sequence $(x(i))_{i<n}$.

	Suppose that $C \subset X_0 \times \dots \times X_n$ for some sets $X_0,\dots, X_n$. For an $i \leq n$ and $(x_0,\dots,x_i) \in X_0 \times \dots \times X_i$ as usual $C_{(x_0,\dots,x_i)}$ will stand for $\{(x_{i+1},\dots,x_n):(x_0,\dots,x_n)\in C\}$, the vertical section of $C$ determined by $(x_0,\dots,x_i)$. We also use the analogous notation for mappings defined on product spaces. $\proj_i$ stands for the projection map $\proj_i: X_0 \times \dots \times X_n \to X_i$.
	
	The standard notations $\mb{\Pi}^0_1(X)$, $\bbo(X)$, $\bs(X), \dots$ will be used for the collection of subsets of $X$ that are closed, Borel, analytic, etc. A coding of the Borel sets with nice properties has to be fixed, let $\mb{BC}(X)$ be a set of Borel codes and sets $\mb{A}(X)$ and $\mb{C}(X)$ with the properties summarized below:

	\begin{fact} (see \cite[3.H]{moschovakis2009descriptive})
		\label{f:prel}
		\begin{itemize}
			\item $\mb{BC}(X) \in \bp(\oom)$, $\mb{A}(X) \in \bs(\oom \times X)$, $\mb{C}(X) \in \bp(\oom \times X)$,
			\item for $c\in \mb{BC}(X)$ and $x \in X$ we have $(c,x) \in \mb{A}(X) \iff (c,x) \in \mb{C}(X)$,
			\item if $P$ is a Polish space and $B \in \bbo(P \times X)$ then there exists a Borel map $f:P \to \oom$ so that $\ran(f) \subset \mb{BC}(X)$ and for every $p \in P$ we have $\mb{A}(X)_{f(p)}=B_p$.
		\end{itemize}
	\end{fact} 
	
	Similarly, there exists a so called \textit{good universal closed set} for every Polish space as well:

	\begin{fact} (\cite[3.H]{moschovakis2009descriptive})
		\label{f:closed}
		There exists a closed set $U^{\mathbf{\Pi}^0_1} \in \mathbf{\Pi}^0_1(\oom \times X)$ so that if $P$ is a Polish space and $C \in \mathbf{\Pi}^0_1(P \times X)$ then there exists a Borel map $f:P \to \oom$ so that for every $p \in P$ we have $U^{\mathbf{\Pi}^0_1}_{f(p)}=C_p$.
		
	\end{fact} 
	
	We will identify a set $x \in \omm$ with its increasing enumeration. As usually, $x \leq^* y$ if $|\{n:x(n) > y(n)\}|<\infty$ and $x \leq y$ holds if for every $n \in \N$ we have $x(n) \leq y(n)$. A set $S \subset \omm$ is \textit{dominating} if for any $y \in \omm$ there exists an $x \in S$ with $y \leq^*x$. We will use the abbreviation $y \leq ^\infty x$ for $|\{n:y(n) \leq x(n)\}|=\infty$.
	
	Notions and facts from effective descriptive set theory will be applied, however, the proofs can be understood using them as ``black boxes". 
	
	If $X$ is a recursively presented Polish space (in our arguments only the spaces $\N, \oom$ and their finite products will appear in such a role) and $p \in \oom$, $\Pi^0_1(X;p)$, $\lbo(X;p)$, $\ls(X;p)$ and $\lp(X;p)$ will stand for the appropriate lightface classes relative to $p$. In the case $X=\om$ the ``$\N$" sometimes will be omitted, and $\lbo(p)$, etc. will be used. If $p \in \oom$ we will denote the first ordinal non-recursive in $p$ by $\omega^p_1$. For an ordinal $\alpha$ and a set $A$ the $\alpha$'s level of the constructible universe relative to $A$ will be denoted by $L_\alpha[A]$. If $p,q \in \oom$, $\Pi^0_1(X;p,q), \omega^{p,q}_1$ etc. will abbreviate the notions $\Pi^0_1(X;\langle p,q \rangle), \omega^{\langle p,q \rangle}_1$ etc., where $\langle \cdot,\cdot \rangle: \oom \times \oom \to \oom$ is a recursive bijection. 
	
	We collect the theorems of the effective theory used in the proof.
	
	\begin{fact}
		\label{f:effective} For any reals $r,p \in \oom$ we have
		\begin{enumerate}
			\item (\cite[Section 3.7]{chong2015recursion}) \label{f:effd} $r \in \Delta^1_1(\om;p) \iff$ $\{r\} \in \Delta^1_1(\oom;p) \iff r \in L_{\omega^p_1}[p]$,
			\item (see below) \label{f:basis} $\omega^{r}_1<\omega^{r,p}_1$ implies that every nonempty $A \in \Sigma^1_1(\oom,r)$ contains an element in $\Delta^1_1(r,p)$,
			\item  \label{f:pi1} (\cite[Theorem 4.1.2]{chong2015recursion}) if $S$ is a nonempty $\Pi^1_1(s)$ set then there exists an $r \in S \cap L_{\omega_1}[s]$.
			\item (Kleene, folklore)
			\label{f:kleene}
			\cite[4D.3-4]{moschovakis2009descriptive}) Suppose that $X$ and $Y$ are recursively presented Polish spaces and $C \subset X \times Y$ is a $\Pi^1_1(p)$ set. Then 
			
			\begin{enumerate}
				\item the set $\{x:(\exists y \in \Delta^1_1(x,p))((x,y)\in C)\}$ is $\Pi^1_1(p)$,
				\item $C$ has a full Borel uniformization if and only if there exists a real $q$ so that for every $x$ we have $\Delta^1_1(x,q) \cap C_x \not = \emptyset$.
			\end{enumerate}
		\end{enumerate}	
	\end{fact}
	
	To see \eqref{f:basis}, using Spector's theorem (\cite[Lemma 2.4.9]{chong2015recursion}) we obtain that $\omega^{r}_1<\omega^{r,p}_1$ implies $\mc{O}^{r} \in \Delta^1_1(r,p)$, where $\mc{O}^{r}$ stands for Kleene's $\mc{O}$ relative to $r$. Thus, as every non-empty $\Sigma^1_1(r)$ set contains a $\Delta^1_1(\mc{O}^{r})$ real by Gandy's basis theorem (\cite[Theorem 2.5.3]{chong2015recursion}), such a set contains a $\Delta^1_1(r,p)$ real as well.

	Let $\mathbf{\Gamma}$ be a family of subsets of Polish spaces. A subset $A$ of a Polish space $X$ (and similarly for a standard Borel space) is \textit{$\mb{\Gamma}$-hard}, if for any $A'\in \mb{\Gamma}$ subset of a Polish space $X'$ there exists a Borel map $f:X' \to X$ with $x' \in A' \iff f(x') \in A$ for every $x'\in X'$. A $\mb{\Gamma}$-hard set that is in $\mb{\Gamma}$ is called \textit{$\mb{\Gamma}$-complete}. For a graph $G=(X,E)$ the \textit{$\mb{\Gamma}$-measurable chromatic number} or \textit{$\mb{\Gamma}$ chromatic number} is defined analogously to the Borel chromatic number with requiring the coloring function to be $\mb{\Gamma}$-measurable, and denoted by $\chi_{\mb{\Gamma}}(G)$. $\chi(G)$ stands for the (usual) chromatic number of the graph $G$. A set $S \subset X$ will be called \textit{independent} or \textit{$E$-independent} if $S^2 \cap E=\emptyset$. Note that if $\mb{\Gamma}$ is closed under finite unions then for $n \in \N$ the existence of a $\mb{\Gamma}$-measurable $n$-coloring of the graph $(X,E)$ is equivalent to the existence of a partition of $X$ to $n$-many $E$-independent sets from $\mb{\Gamma}$.
	
	The Effros Borel space of the closed subsets of a Polish space $X$ will be denoted by $F(X)$. 
	
	\begin{definition}
		\label{d:ond}	
		Let $\bbg$ be a class of subsets of Polish spaces. A map $\Phi:\bbg(X) \to \bp(Y)$ is said to be \emph{$\bp$ on $\bbg$}, if for every Polish space $P$ and $A \in \bbg(P \times X)$ we have $\{(s,y)\in P \times Y:y \in \Phi(A_s)\} \in \mathbf{\Pi}^1_1(P \times Y).$

	\end{definition}
	
	Note that if for some $\Phi$ the above condition holds for $P=\oom$ and $\bbg$ is closed under continuous preimages then $\Phi$ is $\bp$ on $\bbg$: indeed, given $A \in \bbg(P \times X)$ one can fix a continuous map $\phi:\oom \to P$ that is bijective on a closed set $C \subset \oom$, and $\phi|_C$ is a Borel isomorphism (see \cite[Theorem 1G.2]{moschovakis2009descriptive}) and pull back $A$ with the map $(\phi, \id_X)$ to obtain a set $A' \in \bbg(\oom \times X)$. Then using the condition for $\oom$ yields that $\{(s,y)\in \oom \times Y:y \in \Phi(A'_s)\} \in \bp(\oom \times Y)$, so $\{(s,y)\in C \times Y:y \in \Phi(A'_s)\} \in \bp(\oom \times Y)$. But then $(\phi|_C, \id_Y)(\{(s,y)\in C \times Y:y \in \Phi(A'_s)\} )=\Phi(A) \in \bp(P \times Y)$, as $(\phi|_C, \id_Y)$ is a Borel isomorphism.

	\section{General results}
	\label{s:general}
	In this section we prove Theorem \ref{t:main2} about the complexity of uniform families. So, let $\Phi$, $X$, $Y$, $\mathbf{\Gamma}$ be as in the theorem. Before starting the proof we make two easy observations. First, without loss of generality we can assume that $Y=\N^\N$: indeed, composing $\Phi$ with a Borel bijection between $Y$ and $\N^\N$ neither the families $\mc{F}^\Phi$ and $\mc{U}^\Phi$, nor the fact that $\Phi$ is $\bp$ on $\bbg$ changes. So, from now on we assume that $Y=\om^\om$. Second, if one replaces $\N^\N$ by its homeomorphic copy in the definition of nicely $\bs$-hard on $\bbg$ families (Definition \ref{d:nicely}), it yields an equivalent condition on the family $\mc{F}^\Phi$. Throughout the proof we will frequently use this, e.g. saying that ``identify the space $\oom$ with $(\oom)^2$".
	
	Let us roughly sketch the ideas of the proof of Theorem \ref{t:main2}.  Firstly, since $\mc{F}^\Phi$ is $\bs$-hard for a given $r \in \oom$ a diagonal argument yields a set $B \in \bbg(X)$ such that $\Phi(B)$ is nonempty, but contains no $\Delta^1_1(r)$ elements. Also, one can show that the set $\{c:\text{$c$ codes the $\bs$ set $A$ and } \sup \{\omega^{r,c}_1: r\in A\}<\omega_1\}$ is $\mathbf{\Sigma}^1_2$-hard,  our strategy is to reduce this set to the codes of the sets in $\mc{U}^\Phi$. So, to  a given $A \in \bs(\oom)$ with code $c$ using (a uniform version of) the diagonalization we construct a set $B \in \bbg(\oom \times X)$ so that for all $r$ we have $\Phi(B_r) \not = \emptyset$ and $\Phi(B_r) \cap \Delta^1_1(r,c) = \emptyset \iff r \in A$. From this and \eqref{f:kleene} of Fact \ref{f:effective}, it will easily follow that if $\sup \{\omega^{r,c}_1: r\in A\}=\omega_1$ for some $A$ with code $c$ then the corresponding set $B \not \in \mc{U}^{\Phi}$. Finally, the niceness of $\mc{F}^\Phi$ and \eqref{f:basis} of Fact \ref{f:effective} will yield the converse.
	
	We start with the diagonal argument.

	\begin{lemma}
		\label{l:main0} Let $S\in \bp(\N^\N)$ be arbitrary. There exist a real $q \in \N^\om$ and sets $B \in \bbg(\N^\N \times\N^\N \times X)$, $D \in \bs(\N^\N \times\N^\N)$ such that for every $s \in \N^\N$ we have
		\begin{enumerate}
			\item \label{c:main1} $s \in S \iff \forall t \in \N^\N (\Phi(B_{(s,t)}) \cap  \Delta^1_1(q,s,t) \not = \emptyset)$,
			\item \label{c:main2} $\forall t\in \N^\N((\Phi(B_{(s,t)}) \cap  \Delta^1_1(q,s,t) = \emptyset) \implies (D_s\not= \emptyset \land D_s \subset \Phi(B_{(s,t)})))$.
		\end{enumerate}

	\end{lemma}
	
	\begin{proof}
		First we construct an auxiliary set $A$ for the sake of the diagonal argument. Take a universal set $U\in \bs(\N^\N \times \N^\om \times    \N^\om)$ so that for any $R \in \bp(\om^\om \times \N^\om)$ there exists an $r \in \om^\om$ with $U_{r}=(\N^\N \times \N^\om) \setminus R$. Define $(r,s) \in A \iff (r,r,s) \in U$. Then clearly $A \in \bs( \om^\N \times \N^\om)$ and if $R \in \bp(\om^\om \times \N^\om)$ is arbitrary then for some $r \in \N^\N$ we have $(\N^\N \times \N^\om) \setminus R=U_{r}$, so by definition $A_{r}=U_{(r,r)}=\oom \setminus R_r$.
		
		Using the fact that $\mc{F}^\Phi$ is nicely $\bs$-hard (and identifying $\N^\N \times \N^\N$ with $\N^\N$ by a homeomorphism) we obtain sets $B^{0} \in \bbg(\N^\N \times \N^\N \times X)$ and $D^{0} \in \bs(\N^\N \times \N^\N \times \N^\N)$ such that for every $(r,s)$ we have
		\begin{equation}
		\label{e:basic}
		(r,s) \in A \iff \Phi(B^{0}_{(r,s)}) \not = \emptyset \iff D^{0}_{(r,s)} \not =\emptyset\end{equation}
		
		and $D^{0}_{(r,s)} \subset \Phi(B^{0}_{(r,s)})$. Pick also an arbitrary pair $(r^*,s^*) \in A$ and let $q \in \Phi(B^0_{(r^*,s^*)})$. Fix a $C \in
		\mathbf{\Pi}^0_1((\N^\N)^2)$ so that $S=\N^\N \setminus \proj_0(C)$ and define
		
		\[R=\{(r,s) \in \N^\N \times \N^\N:\forall t\in \oom ((s,t) \not \in  C \lor \exists u \in \Delta^1_1(q,s,t)(u \in \Phi(B^{0}_{(r,s)})))\}.\]
		
		Since $\Phi$ is $\bp$ on $\bbg$, the set $\Phi(B^{0})$ is $\bp$, and, using \eqref{f:kleene} of Fact \ref{f:effective} we get $R \in \bp(\om^\om \times \om^\om)$. Then by the choice of $A$ there exists an $r_0 \in \om^\om$ so that $A_{r_0}=\oom \setminus R_{r_0}$.
		
		Define now $(s,t,x) \in B \iff $
		\[((s,t) \in C \text{ and } x \in B^{0}_{(r_0,s)})\text{ or }((s,t) \not \in C \text{ and }
		x \in B^{0}_{(r^*,s^*)})\]
		and $D=D^0_{r_0}$.
		
		We claim that $q$, $B$ and $D$ satisfy the requirements of the Lemma. Indeed, as $\bbg$ is closed under continuous preimages the sets  $\{(s,t,x):x \in B^0_{(r_0,s)} \}$, $\{(s,t,x):x \in B^0_{(r^*,s^*)} \}\in  \bbg$, while using the closedness under finite unions and intersections and $\mathbf{\Pi}^0_1 \cup \mathbf{\Sigma}^0_1 \subset \bbg$ we have $B \in \bbg$. Moreover, $D \in \bs$ by definition.
		
		We check now that \eqref{c:main1} and \eqref{c:main2} hold.

		\textit{Case 1.} $s \in S$. Then for every $t$ we have $(s,t) \not \in C$ so $B_{(s,t)}=B^{0}_{(r^*,s^*)}$. Thus, $q \in \Phi(B_{(s,t)})$, hence $\Delta^1_1(q,s,t) \cap \Phi(B_{(s,t)}) \not = \emptyset.$ So  both \eqref{c:main1} and \eqref{c:main2} hold for every such $s$.
		
		\textit{Case 2.} $s \not \in S.$ Observe that $(r_0,s) \not \in R$: indeed, pick an $(s,t) \in C$, then if $(r_0,s) \in R$ was true, then there would exist a $u \in \Phi(B^0_{(r_0,s)})$. But this is absurd, by \eqref{e:basic} and the fact that $A_{r_0}=\N^\N \setminus R_{r_0}$.
		
		Now suppose that for all $t$ there exists a $u \in \Delta^1_1(q,s,t) \cap \Phi(B_{(s,t)})$. Then, there exists an $(s,t) \in C$ and we have $\Phi(B_{(s,t)})=\Phi(B^0_{(r_0,s)})$, yielding $(r_0,s) \in R$, a contradiction. This completes the proof of \eqref{c:main1}.
		
		In order to see \eqref{c:main2} note that if for some $t$ we have $(s,t) \not \in C$ then $B_{(s,t)}=B^0_{(r^*,s^*)}$, hence $q \in \Phi(B_{(s,t)})$. Thus, it is enough to check \eqref{c:main2} for $(s,t) \in C$ with $\Phi(B_{(s,t)}) \cap \Delta^1_1(q,s,t) = \emptyset$. But then, $B_{(s,t)}=B^0_{(r_0,s)}$ and $D_s=D^0_{(r_0,s)} \subset \Phi(B^0_{(r_0,s)})$. So, what remains to prove is $D^0_{(r_0,s)}=D_s \not = \emptyset$, or by \eqref{e:basic} equivalently $(r_0,s) \in A$, that is, $(r_0,s) \not \in R$, which we already have shown.
		
	\end{proof}
	Now we are ready to prove Theorem \ref{t:main2}.
	
	\begin{proof}[Proof of Theorem \ref{t:main2}]
		Let $A$ be a $\mathbf{\Sigma}^1_2$-complete $\Sigma^1_2$ subset of $\om^\om$ and find $S' \in \Pi^1_1(\om^\om \times \om^\om)$ with $A=\proj_0(S')$. Define
		\[S=\{(q,r,s):\exists u \in \Delta^1_1(q,r,s)((r,u) \in S')\}\]
		and apply Lemma \ref{l:main0} to the set $S$ (after identifying $\N^\N$ with $(\N^\N)^3$). This yields a real $q'$  and sets $B' \in \bbg((\N^\N)^3 \times \oom \times X)$ and $D' \in \bs((\N^\N)^3 \times \oom)$ satisfying \eqref{c:main1} and \eqref{c:main2} of Lemma \ref{l:main0} for every $(q,r,s)$ triple.
		
		Using the fact that $\Phi$ is $\bp$ on $\bbg$ we can find a real $q_0$ so that $\Phi(B') \in \Pi^1_1(q_0)$ and $D' \in \ls(q_0)$, we can also assume that $q' \in \Delta^1_1(q_0)$.
		
		Define $B \in \bbg((\oom)^2 \times X)$ as follows:
		\[(r,\langle s,t\rangle,x) \in B \iff (q_0,r,s,t,x) \in B'\]       
		(recall that $\langle\cdot,\cdot \rangle$ was a recursive homeomorphism between $(\oom)^2$ and $\oom$.)
		
		By \eqref{c:main1} of Lemma \ref{l:main0} for every $(q,r,s) \in (\N^\N)^3$ we have that \begin{equation}
		(q,r,s) \in S \iff \forall t \in \N^\N (\Phi(B'_{(q,r,s,t)}) \cap  \Delta^1_1(q',q,r,s,t) \not = \emptyset),
		\label{e:bpprop}
		\end{equation} hence using the definition of $q_0$ and $B$ we get that
		\begin{equation}
		(q_0,r,s) \in S \iff \forall t \in \N^\N (\Phi(B_{(r,\langle s,t\rangle)}) \cap  \Delta^1_1(q_0,r,s,t) \not = \emptyset).
		\label{e:bprop}
		\end{equation}

		We show that for every $r$ we have $B_r \in \mc{U}^\Phi$ if and only if $r \in A$, which is clearly sufficient to prove the theorem.
		Suppose that $r \not \in A$ and for the contradiction that $B_r \in \mc{U}^\Phi$. By definition this implies that $\Phi(B_r)$ has a full Borel uniformization. Consequently, by Fact \ref{f:effective} there exists a $p_0 \in \N^\om$ such that $\forall s,t \in \N^\N(\Delta^1_1(p_0,s,t) \cap \Phi(B_{(r,\langle s,t\rangle)}) \not = \emptyset),$ so
		\begin{equation}
		\label{e:trivi}
		\forall s, t \in \N^\N(\Delta^1_1(p_0,s,t) \cap \Phi(B'_{(q_0,r,s,t)})\not = \emptyset),
		\end{equation}
		moreover, $q',q_0,r \in \Delta^1_1(p_0)$ can be also assumed. As $S'_r=\emptyset$, by definition $S_{(q_0,r)}$ is also empty. Thus, by \eqref{e:bpprop} for every $s \in \N^\om$ there exists a $t \in \N^\N$ such that $\Delta^1_1(q',q_0,r,s,t) \cap \Phi(B'_{(q_0,r,s,t)})=\emptyset$. Applying this to $s=p_0$ we get a real $t_0$ with $\emptyset=\Delta^1_1(q',q_0,r,p_0,t_0) \cap \Phi(B'_{(q_0,r,p_0,t_0)})=\Delta^1_1(p_0,t_0)\cap \Phi(B'_{(q_0,r,p_0,t_0)})$. This yields a contradiction with \eqref{e:trivi} for $s=p_0, t=t_0$.
		
		Now suppose that $r \in A$.
		
		{\sc Claim.} $\sup \{\omega^{q_0,r,s}_1: s \not \in S_{(q_0,r)}\}<\omega_1$.

		Otherwise, $\bigcup_{s \not \in S_{(q_0,r)}}L_{\omega^{q_0,r,s}_1}[q_0,r] = L_{\omega_1}[q_0,r]$. By Fact \ref{f:effective} \eqref{f:pi1}, as $S'_{r}$ is a nonempty $\Pi^1_1(r)$ set (because $r \in A =\proj_0(S')$) we have that there exists a $u \in S'_{r} \cap L_{\omega_1}[r]$. Then for some $s \not \in S_{(q_0,r)}$ we would have $u \in L_{\omega^{q_0,r,s}_1}[r]\subseteq L_{\omega^{q_0,r,s}_1}[q_0,r,s]$, so $u \in \Delta^1_1(q_0,r,s)$ (by Fact \ref{f:effective} \eqref{f:effd}), contradicting the definition of $S$.

		Now pick a $p$ with $\omega^p_1>\sup \{\omega^{q_0,r,s}_1: s \not \in S_{(q_0,r)}\}$. We will show that $\Phi(B_r)$ has a full Borel uniformization. In order to do this, by Fact \ref{f:effective} \eqref{f:kleene} it is enough to show that for every $s,t \in \N^\N$ we have that $\Phi(B_{(r,\langle s,t\rangle)}) \cap \Delta^1_1(p,q_0,r,s,t) \not = \emptyset$. If $s \in S_{(q_0,r)}$ then this holds by \eqref{e:bprop}.
		
		Now, we can assume that $s \not \in S_{(q_0,r)}$ and $\Phi(B_{(r,\langle s,t\rangle)}) \cap \Delta^1_1(q_0,r,s,t) = \emptyset$, since if the equality is not true then we are already done.
		Then $\Phi(B'_{(q_0,r,s,t)}) \cap \Delta^1_1(q_0,r,s,t) = \emptyset$. Recall the definition of $D'$: it has been obtained from the application of Lemma \ref{l:main0} to the set $S$. Hence, by \eqref{c:main2} of Lemma \ref{l:main0} for every $(q_0,r,s)$ we have that for every $t \in \N^\N$ the implication
		\[(\Phi(B'_{(q_0,r,s,t)}) \cap  \Delta^1_1(q_0,r,s,t) = \emptyset) \implies (D'_{(q_0,r,s)}\not= \emptyset \land D'_{(q_0,r,s)}\subset \Phi(B'_{(q_0,r,s,t)}))\]
		holds. Hence, in our case we have $\emptyset \not =D'_{(q_0,r,s)} \subset \Phi(B'_{(q_0,r,s,t)})=\Phi(B_{(r,\langle s,t \rangle)})$.
		Then $D'_{(q_0,r,s)}$ is a nonempty $\Sigma^1_1(q_0,r,s)$ set and $\omega^{q_0,r,s}_1<\omega^{p,q_0,r,s}_1$ so by Fact \ref{f:effective} \eqref{f:basis} it contains a $\Delta^1_1(p,q_0,r,s)$ real. Thus, $\emptyset \not =\Delta^1_1(p,q_0,r,s,t) \cap D'_{(q_0,r,s)} \subset \Delta^1_1(p,q_0,r,s,t) \cap \Phi(B_{(r,\langle s,t\rangle)})$ which shows that $\Phi(B_r)$ has a full Borel uniformization and finishes the proof of the theorem.
	\end{proof}
	In our theorem the reason of the high complexity is the same phenomenon as in the complexity results of Adams and Kechris \cite{adams2000linear}. In fact, one of their results follows directly from our theorem.
	\begin{corollary} \label{c:kechadams}
		(Adams, Kechris) The set of trees $T$ on $\N \times \N$ so that $[T]$ (the set of the infinite branches of $T$) has a full Borel uniformization is $\mathbf{\Sigma}^1_2$-complete.
	\end{corollary}
	\begin{proof}
	
	    A standard calculation shows that the set in question is $\mathbf{\Sigma}^1_2$, so to show completeness, it is enough to verify that it is $\mathbf{\Sigma}^1_2$-hard. One can check that Theorem \ref{t:main2} can be applied to $X=Y=\oom$, $\bbg=\mathbf{\Pi}^0_2$ and $\Phi(A)=A$. This yields a set $B \in \mathbf{\Pi}^0_2(\oom \times \oom \times \oom)$ so that the set $\{s\in \oom: B_s \in \mc{U}^\Phi\}$ is $\mathbf{\Sigma}^1_2$-hard. Now, one can pick a set $C^0 \in \mathbf{\Pi}^0_1((\oom)^4)$ such that $B$ is the injective projection of $C^0$ to the first three coordinates (see, \cite[1G.5]{moschovakis2009descriptive}). Applying a recursive bijection between $(\N^\N)^2$ and $\N^\N$ to the last two coordinates, one obtains a set $C \in \mathbf{\Pi}^0_1((\oom)^3)$. To every $s \in \oom$ one can continuously assign a tree $T_s$ on $\N \times \N$ so that $[T_s]=C_s$. It is not hard to see that $[T_s]$ has a full Borel uniformization $\iff$ $B_s$ has a full Borel uniformization $\iff$ $B_s \in \mathcal{U}^\Phi$, which shows our claim.
	\end{proof}

	\begin{remark}
		\label{r:niceness} It has been mentioned earlier that the assumption of niceness cannot be dropped from our theorem. On the other hand, assuming $\mathbf{\Sigma}^1_2$-determinacy, it can be, in fact Theorem \ref{t:main1} has a particularly nice form.
		\begin{enumerate}
			\item ($V=L$) There exists a map $\Phi:\mathbf{\Pi}^0_2(\oom) \to \bp(\oom)$ that is $\mathbf{\Pi}^1_1$ on $\mathbf{\Pi}^0_2$, the family $\mc{F}^{\Phi}$ is 
			$\bs$-hard on $\mathbf{\Pi}^0_2$, but $\mc{U}^{\Phi}$ is not $\mathbf{\Sigma}^1_2$-hard.
			\item ($\mathbf{\Sigma}^1_2$-determinacy) Let $X$, $Y$, and $\bbg$ be as in Theorem \ref{t:main1} and $\Phi: \bbg(X) \to \bp(Y)$ be a map that is $\bp$ on $\bbg$. Then if $\mc{F}^\Phi$ is not $\bp$ on $\bbg$ (that is, there exists a set $B \in \bbg(\oom \times X)$ such that $\{s:B_s \in \mc{F}^\Phi\} \not \in \bp$) then $\mc{U}^\Phi$ is $\mathbf{\Sigma}^1_2$-hard.
			
		\end{enumerate}
		$(2)$ can be shown using similar ideas to the ones used in the proof above utilizing Wadge's lemma for higher projective classes, while in $(1)$ one can construct a $\Phi$ so that $\Phi(A)$ can only be nonempty if $|A|\leq \aleph_0$ and moreover every element of $\Phi(A)$ lies very high in the constructible hierarchy. This way it can be assured that for a $B \subset \mathbf{\Pi}^0_2((\oom)^2)$ the set $\Phi(B)$ can be uniformized only if $\proj_1(B)$ is countable. The question whether $\mathbf{\Sigma}^1_2$-determinacy is optimal will be investigated in an upcoming paper.

	\end{remark}

	\section{Consequences on graph colorings}
	\label{s:appl1}
	In this section we apply the results of the previous one and prove Theorem \ref{t:mainintro}. We start with proving a variant of it, from which the full version will be easy to deduce.
	
	\begin{theorem}
		\label{t:main1}
		
		There exists a closed set $C \subset \oom \times \omm$ such that the set $\{s \in \oom:\chi_B(\mc{G}_\mc{S}|_{C_s})<\aleph_0\}$ is $\mathbf{\Sigma}^1_2$-complete.
	\end{theorem}
	
	The next lemma reduces our task to produce a Borel set $B \subset \oom \times \omm$ such that the set $\{s:\chi_B(\mc{G}_\mc{S}|_{B_s})<\aleph_0\}$ is $\mathbf{\Sigma}^1_2$-complete.

	\begin{lemma}
		\label{l:lift}
		Let $B \subset \oom \times \omm$ be a $\mathbf{\Delta}^1_1$ set. There exists a $\mb{\Pi}^0_1$ set $C \subset \oom \times \omm$ and an injective, vertical section preserving (that is, for every $s,x$ we have $\proj_0\Psi(s,x)=s$) continuous map $\Psi:C\to B$ such that for every $s\in \oom$ we have that $\Psi_s,\Psi^{-1}_s$ are shift-invariant maps and if $\chi_B(\mc{G}_\mc{S}|_{B_s}) \geq 3$ then $\chi_B(\mc{G}_\mc{S}|_{B_s})=\chi_B(\mc{G}_\mc{S}|_{C_s})$. Moreover, $\Psi_s$ is a bijection if $B_s$ is closed under the shift.

	\end{lemma}
	
	Note that Theorem \ref{t:mainintro2} also follows from the above lemma: the first statement of it is obvious if $\chi(\mc{G}_\mc{S}|_{B}) \leq 2$, while (a parametrized version of) the rest is Lemma \ref{l:lift}.
	
	\begin{proof}[Proof of Lemma \ref{l:lift}]
		The idea of the proof is that we express $B$ as an injective projection of a closed set. Then, by applying a homeomorphism (that serves as a coding map) to this closed set we will get another closed set so that the composition of the inverse of the projection and the homeomorphism, and also the inverse of this composition are shift-invariant.
		
		Consider the set $B'=\{(s,x): (\forall j \in \N)(\mc{S}^{j}(x) \in B_s)\}$. We will build a set $C$ and a map $\Psi$ as in the lemma, such that $\Psi$ is a bijection between $C$ and $B'$.
		
		Let $(\sigma_i)_{i \in \om}$ be an enumeration of the finite increasing sequences of natural numbers and define $\pred:B' \to [\N]^{<\om}$ by $\pred(s,x)=\{i:(s,\sigma_i \concatt x) \in B'\}$. Since $\pred$ is a Borel map, its graph can be expressed as an injective projection of a closed set in $\mathbf{\Pi}^0_1(\oom \times \omm \times [\N]^{<\om} \times \oom)$ (here $ [\N]^{<\om}$ is endowed with the discrete topology). Let $\overline{\psi}$ be the partial Borel map $\oom \times \omm \times [\N]^{<\om}\to \oom$ so that the graph of $\ovl{\psi}$ is this closed set (see, \cite[1G.5]{moschovakis2009descriptive}).
		
		Given a pair $(s,x) \in B'$ there are finitely many $i$'s with $i \in \pred(s,x)$ and so the set $\{\ovl{\psi}(s,\sigma_i \concatt x,\pred(s,\sigma_i \concatt x)):i \in \pred(s,x)\}$ is finite. Our strategy is to every $(s,x)$ assign a natural number that encodes finite initial segments of the elements of the finite set above. The assigned number to $(s,x)$ should be smaller than the one assigned to $(s,\mc{S}(x))$ and the latter should encode longer initial segments of the corresponding finite set of values (this length will be determined by the first element of $\mc{S}(x)$). We construct a map $\ovl{\Psi}:\oom \times [\om]^\om \to \oom \times [\om]^\om$ by calculating the assigned number to $(s,x), (s,\mc{S}(x)), \dots$, hence guaranteeing the shift-invariance of $\ovl{\Psi}_s$. Finally, we will let $C=\overline{\Psi}(B')$ and $\Psi=\overline{\Psi}^{-1}$, the encoding of longer and longer initial segments will guarantee that the set $C$ is closed.

		More precisely, fix an injection $cd: [\N \times [\om]^{<\om} \times\om^{<\om}]^{<\om} \to \om$ (the set $[\N \times [\om]^{<\om} \times\om^{<\om}]^{<\om}$ is also endowed with the discrete topology). Define $\overline{\Psi}^0(s,x)=$ \[\Big \{\Big(i,\pred(s,\sigma_i \concatt x),\overline{\psi}\big(s,\sigma_i \concatt x,\pred(s,\sigma_i \concatt x)\big)|_{x(0)}\Big): i \in \pred(s,x)\Big\}.\]       
		
		Let $A_n=\{(s,x):(s,x) \in B', x(0)=n\}.$
		In order to achieve the property that the natural numbers serving as codes increase as one applies the map $\mc{S}$, we define a map $\ovl{\Psi}^1$ on the sets $A_n$ inductively. If $\ovl{\Psi}^1$ has been already defined on $\bigcup_{i<n} A_i$ and $(s,x) \in A_{n}$ let $\ovl{\Psi}^1(s,x)=$
		\[2^{x(0)}\cdot 3^{cd(\ovl{\Psi}^0(s,x))}\cdot5^{\max \{0,\ovl{\Psi}^1(s,\sigma_i \concatt x): i \in \pred(s,x), \sigma_i \not = \emptyset\}} . \]
		
		Finally, let \[\ovl{\Psi}(s,x)=(s,(\ovl{\Psi}^1(s,\mc{S}^{j}(x)))_{j\in \mathbb{N}}).\] Note that for each $s$ the map $\ovl{\Psi}(s,\cdot)$ is a Borel map from $B'_{s}$ to $\omm$: by definition $B'_s$ is closed under the shift, so $\ovl{\Psi}(s,\cdot)$ is defined on $B'_s$ and from the definition of $\ovl{\Psi}^1$ it follows that for any $x$ we have \[\ovl{\Psi}^1(s,x)<5^{\ovl{\Psi}^1(s,x)} \leq 5^{\max \{\ovl{\Psi}^1(s,\sigma_i \concatt \mc{S}(x)): i \in \pred(s,\mc{S}(x)), \sigma_i \not = \emptyset\}} \leq \ovl{\Psi}^1(s,\mc{S}(x)).\]

		Observe that an induction on $n$ yields that if $(s,x) \in A_n$ then $\ovl{\Psi}^1(s,x)$ is determined by the values $\{(i,\pred(s,\sigma_i \concatt x),\ovl{\psi}(s,\sigma_i \concatt x,\pred(s,\sigma_i \concatt x))|_m):m \leq x(0), i \in \pred(s,x)\}$. In particular, for a given $k \in \N$ the $k$th coordinate of $\ovl{\Psi}(s,x)$ is determined by $s(k)$ and the set \[
		\begin{split}\{(i,\pred(s,\sigma_i \concatt \mc{S}^k(x)),\ovl{\psi}(s,\sigma_i \concatt \mc{S}^k(x),\pred(s,\sigma_i \concatt \mc{S}^k(x)))|_m):\\i \in \pred(s,\mc{S}^k(x)),m \leq \mc{S}^k(x)(0)\}.
		\end{split}\]

		\begin{claim}
			Suppose that $((s_n,x_n))_{n \in \om}$ is a sequence with elements in $B'$ such that the sequence $(\ovl{\Psi}(s_n,x_n))_{n \in \om}$ is convergent and $i,j \in \N$. Then the sequence $\Big( \big(s_n,\sigma_i \concatt \mc{S}^j(x_n),\pred(s_n,\sigma_i \concatt \mc{S}^j(x_n)),\ovl{\psi}(s_n,\sigma_i \concatt \mc{S}^j(x_n),\pred(s_n,\sigma_i \concatt \mc{S}^j(x_n))\big)\Big)_{n}$ is also convergent in the sense that either for every large enough $n$ we have $i \not \in \pred(s_n, \mc{S}^j(x_n))$, in which case the sequence is eventually not defined, or for every large enough $n$ we have $i \in \pred(s_n, \mc{S}^j(x_n))$ and then it converges.
			\label{c:conv}
		\end{claim}
		
		\begin{proof}
			Clearly, the convergence of the sequence $(\ovl{\Psi}(s_n,x_n))_{n \in \N}$ implies the convergence of $(\ovl{\Psi}^0(s_n,\mc{S}^{j}(x_n)))_{n \in \N}$ and this yields that $\pred(s_n,\mc{S}^j(x_n))$ must stabilize to some set $I$ as $n \to \infty$. If $i \not \in I$ then for every large enough $n$ we have $(s_n,\sigma_i \concatt \mc{S}^j(x_n)) \not \in B'$, hence the sequence is not defined.
			
			Suppose now that $i \in I$, we check the convergence of the required quadruple. The sequence $(s_n)_{n \in \N}$ clearly converges, while the convergence of $(x_n)_{n \in \N}$ follows from the definition of $\ovl{\Psi}^1$ and the fact that for any $j' \in \N$ the sequence $(\ovl{\Psi}^1(s_n,\mc{S}^{j'}(x_n)))_{n \in \N}$ is convergent. The convergence of the third coordinate is implied by the convergence of the sequence $(\ovl{\Psi}^0(s_n,\mc{S}^j(x_n)))_{n \in \N}$. 
			
			Finally, to show the same for the fourth sequence pick an arbitrary $k \in \N$. We check that the values $\ovl{\psi}(s_n,\sigma_i \concatt \mc{S}^j(x_n),\pred(s_n,\sigma_i \concatt \mc{S}^j(x_n)))(k)$ stabilize. Let $j'\geq\max\{k,j\}$. By the convergence of the sequence $(x_n)_n$ we can pick an $i' \in \N$ such that $\sigma_i \concatt \mc{S}^j(x_n)=\sigma_{i'} \concatt \mc{S}^{j'}(x_n)$ holds for every large enough $n$. Since $i \in I$, we have \[\Big(i',\pred(s,\sigma_{i} \concatt \mc{S}^{j}(x_n)),\overline{\psi}\big(s,\sigma_{i} \concatt \mc{S}^{j}(x_n),\pred(s,\sigma_{i} \concatt \mc{S}^{j}(x_n))\big)|_{\mc{S}^{j'}(x_n)(0)}\Big)=\]\[\Big(i',\pred(s,\sigma_{i'} \concatt \mc{S}^{j'}(x_n)),\overline{\psi}\big(s,\sigma_{i'} \concatt \mc{S}^{j'}(x_n),\pred(s,\sigma_{i'} \concatt \mc{S}^{j'}(x_n))\big)|_{\mc{S}^{j'}(x_n)(0)}\Big)\]\[ \in \ovl{\Psi}^0(s_n,\mc{S}^{j'}(x_n)), \] for every large enough $n$. But the sequence $(\ovl{\Psi}^0(s_n,\mc{S}^{j'}(x_n)))_{n \in \N}$ converges, so the values $\overline{\psi}(s,\sigma_{i} \concatt \mc{S}^{j}(x_n),\pred(s,\sigma_{i} \concatt \mc{S}^{j}(x_n)))|_{\mc{S}^{j'}(x_n)(0)}$ must stabilize as well, and then the fact $\mc{S}^{j'}(x_n)(0) \geq j' \geq k$ yields the convergence of the sequence $\big(\ovl{\psi}(s_n,\sigma_i \concatt \mc{S}^j(x_n),\pred(s_n,\sigma_i \concatt \mc{S}^j(x_n)))(k)\big)_{n \in \N}$.

		\end{proof}
		
		Now we check that $\ovl{\Psi}(B')$ is a closed set. In order to see this suppose that $((s_n,x_n))_{n \in \om}$ is a sequence in $B'$ so that $(\ovl{\Psi}(s_n,x_n))_{n \in \N}$ is convergent. By Claim \ref{c:conv} the sequence $(s_n,x_n, \pred(s_n,x_n),\ovl{\psi}(s_n,x_n, \pred(s_n,x_n)))_n$ is convergent, and by the fact that the graph of $\ovl{\psi}$ is closed, it converges to some $(s,x,\pred(s,x),\ovl{\psi}(s,x, \pred(s,x)))$. To see that $\ovl{\Psi}(s_n,x_n) \to \ovl{\Psi}(s,x)$ holds, pick an arbitrary $k \in \N$. We show that $\ovl{\Psi}(s_n,x_n)(k) \to \ovl{\Psi}(s,x)(k)$. Using the convergence of $(s_n,\mc{S}^j(x_n),\pred(s_n,\mc{S}^j(x_n)))_{n \in \N}$ for every $j \in \N$ (which follows from Claim \ref{c:conv}) we can assume that for each $n$ and $j \leq k$ we have $\mc{S}^j(x_n)(0)=\mc{S}^j(x)(0)$ and $\pred(s_n,\mc{S}^j(x_n))=\pred(s,\mc{S}^j(x))$. By Claim \ref{c:conv} for any $j \leq \mc{S}^k(x)(0)=\mc{S}^k(x_n)(0)$ and $i \in \pred(s,\mc{S}^j(x))=\pred(s,\mc{S}^j(x_n))$ the sequence $\Big(\big(s_n,\sigma_i \concatt \mc{S}^j(x_n),\pred(s_n,\sigma_i \concatt \mc{S}^j(x_n)),\ovl{\psi}(s_n,\sigma_i \concatt \mc{S}^j(x_n),\pred(s_n,\sigma_i \concatt \mc{S}^j(x_n))\big)\Big)_{n}$ is convergent and again by the closedness of the graph of $\ovl{\psi}$ its limit is necessarily $\big(s,\sigma_i \concatt \mc{S}^j(x),\pred(s,\sigma_i \concatt \mc{S}^j(x)),\ovl{\psi}(s,\sigma_i \concatt \mc{S}^j(x),\pred(s,\sigma_i \concatt \mc{S}^j(x)))\big)$. Hence, using the observation made before Claim \ref{c:conv} for a large enough $n$ all the values determining $\ovl{\Psi}(s_n,x_n)(k)$ and $\ovl{\Psi}(s,x)(k)$ will be the same. Thus $C=\ovl{\Psi}(B')$ is indeed closed.
		
		Note that for any $(s,y)=\ovl{\Psi}(s,x)$ and $j \in \N$ the exponent of $2$ in $y(j)$ is $x(j)$. Hence, $\ovl{\Psi}$ is invertible and $\ovl{\Psi}^{-1}(s,\mc{S}(y))=(s,\mc{S}(x))$. Thus, we obtain that $\Psi:=\ovl{\Psi}^{-1}$ is a continuous bijection, so that for each $s \in \oom$ the map $\Psi_s$ back-and-forth shift-invariant. This implies that $\chi_B(\mc{G}_{\mc{S}}|_{B'_s})=\chi_B(\mc{G}_{\mc{S}}|_{C_s})$.

		Finally, we turn back to the set $B$. Of course, if $B_s$ is closed under the shift then $B'_s=B_s$, hence the only thing left to check from the lemma is that whenever $\chi_B(\mc{G}_{\mc{S}}|_{B_s}) \geq 3$ then $\chi_B(\mc{G}_{\mc{S}}|_{B_s})=\chi_B(\mc{G}_{\mc{S}}|_{B'_s})$. Clearly, $\chi_B(\mc{G}_{\mc{S}}|_{B_s}) \geq \chi_B(\mc{G}_{\mc{S}}|_{B'_s})$. For a given $c_0:[\mathbb{N}]^\mathbb{N} \to n$ Borel $n$-coloring of $\mc{G}_{\mc{S}}|_{B'_{s}}$  define\\ $c(x)=$
		\[\begin{cases}
		\min \{m:\mc{S}^{m}(x) \not \in B_s\} \mod 2, & \text{if for all $k$ we have $\mc{S}^k(x) \not \in B'_s$,}\\
		c_0(\mc{S}^k(x))+k \mod n, & \text{otherwise, where $k$ is minimal with $\mc{S}^k(x) \in B'_s$}.
		\end{cases} \]
		
		It is not hard to check that $c$ is a Borel $\max\{2,n\}$-coloring of $\mc{G}_{\mc{S}}|_{B_s}$, hence $\chi_B(\mc{G}_{\mc{S}}|_{B_s}) \geq 3$ implies $\chi_B(\mc{G}_{\mc{S}}|_{B_s})=\chi_B(\mc{G}_{\mc{S}}|_{B'_s})$.

	\end{proof}

	Thus, in order to show Theorem \ref{t:main1} it is enough to construct the required Borel set. This will be done in two steps. Let $\mc{H}=(\oom \times \omm,E_\mc{H})$ where $(s,x)E_\mc{H}(s',x') \iff s=s'$ and $xE_\mc{S}x'$. First we will notice that $\mc{G}_\mc{S}$ contains an isomorphic copy of $\mc{H}$ and then using Theorem \ref{t:main2} we will show that the finitely chromatic Borel subsets of the graph $\mc{H}$ are already $\mathbf{\Sigma}^1_2$-hard.
	
	\begin{lemma}
		\label{l:emb} There exists a continuous injection $e:\mathbb{N}^\mathbb{N} \times [\mathbb{N}]^\mathbb{N} \to [\mathbb{N}]^\mathbb{N}$ that is an isomorphism between $\mc{H}$ and $\mc{G}_{\mc{S}}|_{\ran(e)}$.
	\end{lemma}
	\begin{proof}
		
		Fix a continuous injection $e^0:\oom \to \mc{A}$ such that $\mc{A} \subset \omm$ is an almost disjoint family. For $(s,x) \in \oom \times \omm$ let $e(s,x)= e^0(s) \circ x (=(e^0(s)(i))_{i \in x})$. All the required properties of $e$ are clear  from the fact that $\mc{A}$ is an almost disjoint family.
	\end{proof}
	
	We will use an observation of Di Prisco and the first author that says that the restrictions of the shift graph to non-dominating subsets of $\omm$ have finite Borel chromatic number. In the latter part of the paper a uniform version of this statement is needed, so for the sake of completeness we include a proof of the uniform version. We will use the sets $\mb{BC},\mb{A},\mb{C}$ from Fact \ref{f:prel}. Fix also a homeomorphism $\langle \cdot,\cdot,\cdot \rangle: (\oom)^3 \to \oom$.

	\begin{lemma} (Di Prisco, Todor\v{c}evi\'c, \cite{di2015basis})
		\label{l:nondom} There exists a Borel function $f_{dom}: \omm \to \oom$ so that for each $x\in \omm$ we have $f_{dom}(x)=\langle c_0,c_1,c_2\rangle$ with $c_i \in \mb{BC}(\omm)$, $\mb{A}(\omm)_{c_i}$ are $E_{\mc{S}}$-independent subsets of $\omm$ for every $i$ and $\{y:y \leq^\infty x\} = \bigcup^2_{i=0} \mb{A}(\omm)_{c_i}$.   
	\end{lemma}
	
	\begin{proof}
	    Set $D=\{(x,y):y \leq^\infty x\}.$ Note that it suffices to construct a Borel map $c:D \to 3$ that is a coloring of the graph $\mathcal{G}_\mathcal{S} |_{D_x}$ for each $x$: indeed, we can use Fact \ref{f:prel} for $B_i=\{(x,y):c(x,y)=i\}$ to obtain Borel maps $f_i:\omm \to \oom$ so that for every $x \in \omm$ we have $\mb{A}(\omm)_{f_i(x)}=(B_i)_x$ and let $f_{dom}(x)=\langle f_0(x),f_1(x),f_2(x)\rangle$.
	    
	     We will construct a Borel set $U \subset (\omm)^2$ such that whenever $(x,y) \in D$, then there exists a $i \in \mathbb{N}$ with $(x,\mathcal{S}^i(y)) \in U$ and for each $x$ the set $\{y:(x,y) \in D \cap U\}$ is $\mathcal{G}_\mathcal{S}$-independent. This is enough, as
	     the map 
	     \[c(x,y)=
	     \begin{cases}
	             i \mod 2, \text{ if } (x,y) \not \in U \text{ and $i$ is minimal with $(x,\mathcal{S}^i(y)) \in U$}\\
	             2, \text{ if } (x,y) \in U
	        \end{cases}\]
	    is an appropriate coloring.
	    
	   Now, let $x(-1)=0$ and define $(x,y)\in U \iff$ for the minimal $n \in \mathbb{N}$ with $|[x(2n-1),x(2n+1)) \cap y| \neq 0$ we have that $|[x(2n-1),x(2n+1)) \cap y|$ is even.
	   Observe that if $y \leq^\infty x$ then there exist infinitely many $n$'s such that \[|[x(2n-1),x(2n+1)) \cap y| \geq 2.\] From this, one easily checks that $U$ satisfies the requirements.

	\end{proof}
	
	\begin{lemma}
		\label{l:forH}
		There exists a $\mathbf{\Pi}^0_2$ set $B \subset \oom \times \oom \times \omm$ so that the set $\{s: \chi_B(\mc{H}|_{B_s})<\aleph_0\}$ is $\mb{\Sigma}^1_2$-hard.
	\end{lemma}
	
	\begin{proof}
		We check the applicability of Theorem \ref{t:main2}, with $X=\omm$, $Y=\oom$, $\bbg=\mathbf{\Pi}^0_2$ and  \[\Phi(A)=\{c:(\forall x,y \in A)\Big(c=\langle c_0,c_1,c_2\rangle \text{, $c_i \in \mb{BC}(\omm)$}, x \in \bigcup_i  \mb{A}(\omm)_{c_i} \]\[ \text{and }xE_{\mc{S}}y \Rightarrow (\forall i)\big(\lnot(x,y \in \mb{A}(\omm)_{c_i})\big)\Big) \},\]
		in other words, $\Phi(A)$ contains the Borel codes of the Borel 3-colorings of $A$.
		First we show that $\Phi$ is $\mathbf{\Pi}^1_1$ on $\mathbf{\Pi}^0_2$. If $B$ is a $\mathbf{\Pi}^0_2$ subset of $\oom \times \omm$, then \[\bar{\Phi}(B)=\{(s,c):(\forall (x,y) \in \omm \times \omm)\Big((s,x)\not \in B \text{ or }(s,y) \not \in B \text{ or } (c=\langle c_0,c_1,c_2\rangle, \]\[  \text{$c_i \in \mb{BC}(\omm)$, $x \in \bigcup_i  \mb{C}(\omm)_{c_i}$ and } xE_{\mc{S}}y \Rightarrow (\forall i)\big(\lnot(x,y \in \mb{A}(\omm)_{c_i})\big)\Big) \},\]
		which set is $\bp$.
		
		Now, we show that $\mc{F}^\Phi$ is nicely $\bs$-hard on $\mathbf{\Pi}^0_2$. Let $A \subset \oom$ be analytic and take a closed set $F \subset \oom \times \omm$ so that $\proj_0(F)=A$. Let \[B=\{(s,y):(\forall x \leq^* y)(x \not \in F_s)\}.\]
		We show that the complement of $B$ is $\mathbf{\Sigma}^0_2$, hence $B \in \mathbf{\Pi}^0_2$. For every $\sigma \in \mathbb{N}^\mathbb{N}$ that is eventually zero define $B'_{\sigma}=\{(s,y):(\exists x \leq y+\sigma)(x \in F_s)\}$. Clearly, $(\oom \times \omm) \setminus B=\bigcup_{\sigma} B'_{\sigma}$, so it is enough to show that each $B'_{\sigma}$ is closed.
		Let $((s_m,y_m))_{m \in \om}\subset B'_{\sigma}$ and suppose that $(s_m,y_m) \to (s,y)$. Then for each $m$ there exists an $x_m \leq y_m+\sigma$ so that $(s_m,x_m) \in F$. For every fixed $n$ we have that the set $\{m:(\exists k \leq n)(x_m(k) > y(k)+\sigma(k))\}$ is finite. Thus, applying K\"onig's Lemma to the tree formed by $\{x_m|_k:k,m\in \mathbb{N}, x_m|_k \leq (y+\sigma)|_k\}$ we get that $(x_m)_{m \in \om}$ contains a convergent subsequence, and its limit witnesses $(s,y)\in B'_{\sigma}$. Let \[D=\{(s,c):s \in A \text{ and }(\exists x \in F_s) (f_{dom} (x)=c)\},\] where $f_{dom}$ is the function from Lemma \ref{l:nondom}. We will show that $B$ and $D$ witness that $\mc{F}^\Phi$ is nicely $\mathbf{\Sigma}^1_1$-hard. We have already seen that $B\in \mathbf{\Pi}^0_2$ and by definition $D$ is analytic. \\           
		Suppose that $s \in A$. Then for each $x' \in F_s$ we have $B_s(= \{y:(\forall x \leq^* y)(x \not \in F_s)\}) \subset \{y:y \leq^\infty x'\}$. Thus, by Lemma \ref{l:nondom} $B_s \in \mc{F}^\Phi$ and $D_s \not = \emptyset$. Moreover, if $c \in D_s$ then for some $x \in F_s$ we have $f_{dom}(x)=c$ with $c=\langle c_0,c_1,c_2\rangle$, again by Lemma \ref{l:nondom} we have $B_s \subset \{y:y \leq^\infty x\} = \bigcup^2_{i=0} \mb{A}(\omm)_{c_i}$  and the sets $\mb{A}(\omm)_{c_i}$ are $E_\mc{S}$-independent, thus, $D_{s} \subset \Phi(B_s)$. Now, if $s \not \in A$ then $F_s=D_s=\emptyset$ and $B_s=\omm$.
		
		So, Theorem \ref{t:main2} is applicable and it yields a Borel set $B \subset \oom \times \oom \times \omm$ so that $\{s:B_s \in \mc{U}^\Phi\}$ is $\mb{\Sigma}^1_2$-hard.
		
		Now we claim that $B_s \in \mc{U}^\Phi$ is equivalent to $\chi_B(\mc{H}|_{B_s})<\aleph_0$. Suppose first that for some $s\in \oom$ we have $\chi_B(\mc{H}|_{B_s})<\aleph_0$. Then, by Lemma \ref{l:emb} $\mc{H}|_{B_s}$ is Borel isomorphic to a subgraph of $\mc{G}_\mc{S}$, so if it has finite Borel chromatic number then it has one $\leq 3$ by \cite{kechris1999borel}. Let $S_0,S_1,S_2$ witness this fact. Using Fact \ref{f:prel} there are Borel maps $f_0,f_1,f_2$ so that for any $t \in \oom$ we have that $f_i(t) \in \mb{BC}(\omm)$ and $\mb{A}(\omm)_{f_i(t)}=(S_i)_{t}$. Clearly, $f=\langle f_0,f_1,f_2 \rangle$ is a Borel uniformization of $\Phi(B_{s})$. For the converse suppose that $\Phi(B_s)$ has a Borel uniformization, $f$. Define $S_i=\{(t,x):x \in \mb{A}(\omm)_{f(t)(i)}\} (=\{(t,x):x \in \mb{C}(\omm)_{f(t)(i)}\})$. The sets $S_i$ are Borel, and for each $t$ the sets $(S_i)_{t}$ form a $3$-coloring of $\mc{H}|_{B_{(s,t)}}$, so by the definition of $\mc{H}$ the sets $S_i$ form a Borel $3$-coloring of $\mc{H}|_{B_s}$.

	\end{proof}
	
	\begin{proof}[Proof of Theorem \ref{t:main1}]
		Consider $e$ from Lemma \ref{l:emb} and apply $(id_{\oom},e)$ to the Borel set given by Lemma \ref{l:forH}. This yields a Borel set $B \subset \oom \times \omm$ so that $\{s:\chi_B(\mc{G}_{\mc{S}}|_{B_s})<\aleph_0\}$ is $\mb{\Sigma}^1_2$-hard. Applying Lemma \ref{l:lift} to this set we get a closed set $C \subset \oom \times \omm$ so that the set $\{s:\chi_B(\mc{G}_\mc{S}|_{C_s})<\aleph_0\}$ is $\mathbf{\Sigma}^1_2$-hard. In order to see that this set is $\mb{\Sigma}^1_2$, similarly to the proof of Lemma \ref{l:forH} just notice that $\{s:\chi_B(\mc{G}_\mc{S}|_{C_s})<\aleph_0\}=$
		\[\{s:(\exists c_0, c_1, c_2)(c_i \in \mb{BC}(\omm),( \forall x,y \in C_s)( x \in \bigcup_i  \mb{C}(\omm)_{c_i} \tag{*}\]\[ \text{and }xE_{\mc{S}}y \Rightarrow (\forall i)(\lnot (x,y \in \mb{A}(\omm)_{c_i} ))) \}. \]       
	\end{proof}
	
	\begin{remark}
		In the proof of Lemma \ref{l:forH} we actually show that the collection of non-dominating $\mathbf{\Pi}^0_2$ sets is $\mb{\Sigma}^1_1$-hard in the codes. The proof presented here is an alternate non-effective version of an unpublished result of Hjorth \cite{hjorth}. A similar argument has been also used by Solovay \cite{solovay}. We would like to mention here that more is true: even the collection of non-dominating closed sets is $\mb{\Sigma}^1_1$-hard in the codes.
	\end{remark}
	
	We conclude this section with proving our main result, Theorem \ref{t:mainintro}. In order to formulate the precise statement we use the set $U^{\mb{\Pi}^0_1}$ for $X=\omm$ from Fact \ref{f:closed}.
	
	\begin{theorem}[Theorem \ref{t:mainintro}]
		The collection of closed subsets of $\omm$ so that $\chi_B(\mc{G}_\mc{S}|_C)<\aleph_0$ (or, equivalently, $\chi_B(\mc{G}_\mc{S}|_C)\leq 3$) is $\mathbf{\Sigma}^1_2$-complete, more precisely, the sets
		
		\begin{enumerate}
			\item \label{st:univ} $\{x \in \oom:\chi_B(\mc{G}_\mc{S}|_{U^{\mathbf{\Pi}^0_1}_x})<\aleph_0\}$
			\item \label{st:effros}$\{C \subset \omm: C \text{ is closed, } \chi_B(\mc{G}_\mc{S}|_{C})<\aleph_0\}$ as a subset of the Effros Borel space,
			
		\end{enumerate}
		are $\mathbf{\Sigma}^1_2$-complete.
		
		\begin{enumerate}
			\item[(3)] \label{st:con} Consequently, there is no sequence of $\mathbf{\Sigma}^1_1$-parametrizable collections of graphs $(\mc{A}_n)_{n \in \mathbb{N}}$ so that for every $C \subset \omm$ closed $\chi_B(\mc{G}_\mc{S}|_C)\geq \aleph_0$ if and only if $\exists (n_i)_{i\in \mathbb{N}}$ and  $A_{n_i} \in \mc{A}_{n_i}$ so that $A_{n_i} \leq_B \mc{G}_\mc{S}|_C$. In particular, there is no one element basis, or countable basis in the sense of Question \ref{q:kech}.
		\end{enumerate}
	\end{theorem}
	
	In order to show the statement that talks about the closed sets with the Effros Borel structure we state a general lemma which essentially follows from the work of Sabok \cite{sabok2012complexity}.

	\begin{lemma}
		\label{l:sabok}
		Suppose that $P$ is a property of closed subsets of $\oom$ and there exists a closed set $C \subset \oom \times \oom$ so that $\{x \in \oom:C_x \text{ has } P \}$ is $\mathbf{\Sigma}^1_2$-hard. Then $\{F \subset \oom:F \text{ has } P\}$ is also $\mb{\Sigma}^1_2$-hard as a subset of the Effros Borel space.
	\end{lemma}
	
	\begin{proof}
		Consider the map $f:\oom \to F(\oom)$ given by $f(x)=C_x$. As usual, we can identify $F(\oom)$ with the collection of pruned trees on $\om^{<\om}$, hence it becomes a Borel subset of $2^{\om^{<\om}}$, let us endow $F(\oom)$ with the inherited topology. By \cite[Theorem 2]{sabok2012complexity} it is enough to show that $f$ is $\mb{\Sigma}^1_1 \cup \bp$-submeasurable, that is, there exists a subbase $\mc{B}$ of $F(\oom)$ so that for any $U \in \mc{B}$ we have that $f^{-1}(U) \in \mb{\Sigma}^1_1(\oom) \cup \bp(\oom)$. A subbase for this space can be given in the form $\{F \in F(\oom):F \cap [\sigma]\not = \emptyset\}$ and  $\{F \in F(\oom):F \cap [\sigma] = \emptyset\}$, where $\sigma \in \N^{n}$ and $[\sigma]=\{r \in \N^\N: \sigma =r|_n\}$ for some $n \in \N$. Clearly, for each $\sigma$ the set $f^{-1}(\{F:F \cap [\sigma]= \emptyset\})$ is $\bp$, hence, it is enough to show the set $f^{-1}(\{F:F \cap [\sigma]\not = \emptyset\})=\{x:C_x \cap [\sigma]\not = \emptyset\}=\{x:(\exists y)((x,y)\in C, y \in [\sigma])\}$ is $\mb{\Sigma}^1_1$, which is obvious.
		
	\end{proof}
	
	\begin{proof}[Proof of Theorem \ref{t:mainintro}]
		In \eqref{st:univ} and \eqref{st:effros} the fact that the sets are $\mb{\Sigma}^1_2$ can be easily seen directly, similarly to (*). Moreover, Lemma \ref{l:sabok} and the fact that $\oom$ is homeomorphic to $\omm$ shows that \eqref{st:univ} implies \eqref{st:effros}.
		
		Take the set $C$ from Theorem \ref{t:main1}. For \eqref{st:univ} notice that by Fact \ref{f:closed} there exists a Borel map $f:\oom \to \oom$ so that for any $s$ we have $U^{\mb{\Pi}^0_1}_{f(s)}=C_s$. Then $\{s \in \oom:\chi_B(\mc{G}_\mc{S}|_{C_s})<\aleph_0\}=f^{-1}(\{s \in \oom:\chi_B(\mc{G}_\mc{S}|_{U^{\mathbf{\Pi}^0_1}_s})<\aleph_0\}),$ which shows that the latter set is $\mathbf{\Sigma}^1_2$-complete.
		
		For the last statement, suppose that such collection of $\mc{A}_i$'s exists with the appropriate parametrizations $E_i \subset X_i \times Y^2_i$. Then,    
		$\{s \in \oom:\chi_B(\mc{G}_\mc{S}|_{C_s})<\aleph_0\}=\{s \in \oom: (\forall (n_i)_{i \in \N} \text{ sequence of naturals})(\exists i \in \N)(\forall c \in \oom)(\forall x \in X_{n_i})$
		
		\begin{enumerate}[label=(\alph*)]
			\item \label{c:isol} $(Y_{n_i},(E_{n_i})_x)$ is not Borel isomorphic to a graph in $\mc{A}_{n_i}$ or
			\item \label{c:codesl} $c \not \in \mathbf{BC}(Y_{n_i} \times \omm)$ or

			\item \label{c:nontrivl} $\exists y \in Y_{n_i}$ so that $\lnot ((\exists ! z)((y,z)\in \mathbf{A}(Y_{n_i} \times \omm)_c))$ or
			
			\item \label{c:homoml} $\exists (y,y')\in (E_{n_i})_x,\exists z,z' \in \omm$ so that $(y,z)$, $(y',z') \in \mathbf{A}(Y_{n_i} \times \omm)_c$ and $(z,z') \in C^2_s \setminus E_\mc{S}\}$.
		\end{enumerate}  
		Note that \ref{c:nontrivl} and \ref{c:homoml} express that $c$ does not code a total function and that the function coded is not a homomorphism from $(Y_{n_i},(E_{n_i})_x)$ to $\mathcal{G}_\mathcal{S}|_{C_s}$. Clearly, the formula \ref{c:isol} is $\bp$ and the formulas \ref{c:codesl} and \ref{c:homoml} are $\mathbf{\Sigma}^1_1$. It follows from the (uniform version of) Luzin's unicity theorem \cite[Theorem 18.11]{kechrisbook} that \ref{c:nontrivl} is $\bs$ (an alternative proof can be given using \eqref{f:kleene} of Fact \ref{f:effective}).
		
		Consequently, the existence of the families $(\mc{A}_n)_{n \in \N}$ would imply that the set $\{s \in \oom:\chi_B(\mc{G}_\mc{S}|_{C_s})<\aleph_0\}$ is $\mathbf{\Pi}^1_2$, contradicting Theorem \ref{t:main1}.
	\end{proof}
	\section{Relation to Hedetniemi's conjecture and open problems}
	
	In this section we collect several open problems and discuss the relation of our results to Hedetniemi's conjecture.
	
	Let $\mc{G}=(X,E)$ and $\mc{G'}=(X',E')$ be Borel graphs. The \textit{product of the graphs $\mc{G}$ and $\mc{G}'$}, $\mc{G} \times \mc{G'}$, is
	the graph $(X\times X', E_{\mc{G} \times \mc{G'}})$, where $(x,x') E_{\mc{G} \times \mc{G'}} (y,y') \iff  (xEy\text{ and } x'E'y')$. It is clear that  $\mc{G} \times \mc{G'}$ is a Borel graph and note also that $\chi_B(\mc{G} \times \mc{G'}) \leq \min \{\chi_B(\mc{G}),\chi_B(\mc{G}')\}$. (The Borel version of) Hedetniemi's conjecture is the statement \[\chi_B(\mc{G} \times \mc{G'}) = \min \{\chi_B(\mc{G}),\chi_B(\mc{G}')\}.\]

	The classical Hedetniemi's conjecture is the above statement for finite graphs (and thus with usual chromatic numbers). Clearly, the Borel version of the conjecture for graphs with finite Borel chromatic numbers implies the classical one. However, there are substantial differences between the Borel and classical cases for infinite chromatic numbers.
	
	On the one hand, note that if for some graph $\mc{G}''$ we have $\mc{G}'' \leq_B \mc{G}, \mc{G'}$ then $\mc{G}'' \leq_B \mc{G} \times \mc{G'}$, thus, in such a situation $\chi_B(\mc{G}'')$ gives a lower bound for the value $\chi_B(\mc{G} \times \mc{G'})$. For instance, the $\mathbb{G}_0$ dichotomy implies  Hedetniemi's conjecture for analytic graphs of Borel chromatic number $>\aleph_0$.
	
	On the other hand, it has been proved by Hajnal \cite{hajnal} that there exist graphs $G$ and $G'$ so that $\chi(G)=\chi(G')=\aleph_1$, but $\chi(G \times G')<\aleph_1$. Moreover, it has been shown in \cite{kechris1999borel} that it is consistent that there exist graphs $\mc{G},\mc{G}'$ with coanalytic edge relation such that $\chi_B(\mc{G}),\chi_B(\mc{G}')>\aleph_0$, but $\chi_B(\mc{G}\times \mc{G}') \leq \aleph_0$.
	
	Note also that a compactness argument implies that if $\chi(G)=\aleph_0$ and $\chi(G')=n$ then the conjecture holds.

	Concerning the conjecture for finite graphs\footnote{In a more recent development, Shitov \cite{shitov} gave a counterexample to Hedetniemi's conjecture. This of course gives a counterexample for the Borel version as well.} it is known that for any $n>2$ there are graphs with chromatic number $n$ and arbitrarily high odd girth, thus, there is no finite graph $H$ with $\chi(H)=n$ that would admit a homomorphism to each finite $G$ with $\chi(G) \geq n$. So, Hedetniemi's conjecture cannot be solved by a basis result in the collection of finite graphs (see e. g. \cite{sauer2001hedetniemi}). However, we would like to remark that the finite conjecture is in fact equivalent to a basis result if we are allowed to consider infinite graphs and the right notion of chromatic number:
	
	\begin{remark}
		\label{r:hedet}
		Let $(G_i)_{i \in \om}=((V_i,E_i))_{i \in \om}$ be an enumeration of all the finite graphs with chromatic number $n$. Let $\mc{G}_\infty$ be their infinite product, that is, $\mc{G}_\infty=(\prod_i V_i, E_{\prod_i G_i})$, where $(v_0,v_1,\dots )E_{\prod_i G_i} (v'_0,v'_1,\dots )$ if and only if for every $i \in \om$ we have $v_iE_iv'_i$. $\mc{G}_\infty$ is a Borel graph with a closed edge relation and it is not hard to see that Hedetniemi's conjecture for $n$ implies that $\mc{G}_\infty$ has clopen chromatic number $n$. Conversely, since $\mc{G}_\infty$ admits a continuous homomorphism into each $G_i$, if the clopen chromatic number of $\mathcal{G}_\infty$ is $n$, then Hedetniemi's conjecture holds for $n$.
	\end{remark}
	
	Unfortunately, it is not clear whether it is possible to turn antibasis results to counterexamples to the Borel version of Hedetniemi's conjecture. But, if one considers $\Delta^1_1$-measurable colorings instead of Borel ones our construction yields an example, which works for a rather simple reason: there exist $\Delta^1_1$ sets $B$ and $C$ so that $\mc{G}_\mc{S}|_B$ has a finite Borel chromatic number but has no finite $\Delta^1_1$ coloring and $C$ contains reals which code finite Borel colorings of $B$, hence the $\Delta^1_1$ chromatic number of the product graph will be finite.
	\begin{proposition}
		\label{p:hedet}
		There exist sets $B,C \in \Delta^1_1(\omm)$ so that $\chi_B(\mc{G}_\mc{S}|_C)=\aleph_0$,  $\mc{G}_\mc{S}|_B$ has no $\Delta^1_1$ finite coloring, but the product $\mc{G}_\mc{S}|_B \times \mc{G}_\mc{S}|_C$ has a $\Delta^1_1$ 3-coloring.
	\end{proposition}
	\begin{proof}[Proof sketch]
		Instead of constructing a set $C \subset \omm$ we construct a set $C \subset \oom \times \omm$ and prove that $\mc{H}|_{C}$ and $\mc{G}_\mc{S}|_B$ has the required properties, from this it is easy to deduce the proposition using Lemma \ref{l:emb}.
		
		Pick a set $A \in \Sigma^1_1(\om) \setminus \Pi^1_1(\om)$ and a set $C'' \in \Pi^0_1(\om \times \omm)$ so that $\proj_{0}(C'')=A$. Let \[B''=\{(n,y):(\forall x \leq^* y)(x \not \in C''_n)\}.\]   
		It is not hard to check (similarly to the proof of Lemma \ref{l:forH}) that $B'' \in \Delta^1_1(\om \times \omm)$, the set $\{n: \mc{G}_\mc{S}|_{B''_n} \text{ has a finite}$ $\text{ $\Delta^1_1$ coloring}\} $ is $\Pi^1_1$ and $A$ contains this set. Consequently, for some $n \in A$ we have that $B''_n$ has no $\Delta^1_1$ finite coloring, let $C=\{(x,r):x \in C''_n, r\in \omm\}$ and $B=B''_n$. Then clearly $C \in \Delta^1_1(\oom \times \omm)$ and as $C''_n$ is nonempty, $\chi_B(\mc{H}|_C)=\aleph_0$. Note now that for every $(x,r) \in C$ clearly $B \subset \{y:y \leq^\infty x\}$. Thus, by (the lightface version of) Lemma \ref{l:nondom} for each $(x,r) \in C$ the graph $\mc{G}|_B$ has a $\Delta^1_1(x,r)$ $3$-coloring. Using \ref{f:kleene} of Fact \ref{f:effective}, the graph $\mc{H}|_{C} \times \mc{G}_\mc{S}|_{B}$ has a $\Delta^1_1$ $3$-coloring: we can construct a coloring from the $\Delta^1_1(x,r)$-colorings uniformly.
	\end{proof}
	As we have seen, Theorem \ref{t:mainintro} excludes the possibility of a simple Borel/analytic basis. However, the following is still possible:
	\begin{question}
		Does there exist a graph $\mc{G}=(X,E)$ where $X$ is a Polish space and $E$ is a $\mathbf{\Pi}^1_1$ edge relation so that for any Borel graph $\mc{G}'$ we have $\chi_B(\mc{G}') \geq \aleph_0$ if and only if $\mc{G} \leq_B \mc{G}'$?
	\end{question}
	
	Note that the above question makes sense even with some finite number instead of $\aleph_0$. A possibility of a positive answer is even more intriguing in the light of Remark \ref{r:hedet}: it would be very interesting if in both cases the large chromatic number of a certain class of graphs was witnessed by a graph outside of this class.

	On the other hand we don't know whether the idea of Proposition \ref{p:hedet} can be turned to a counterexample to the Borel version Hedetniemi's conjecture.
	
	\begin{question} Do there exist Borel subgraphs $\mc{G}, \mc{G}'$ of $\mathcal{G}_S$ so that $\chi_B(\mc{G}\times \mc{G}')<\min\{\chi_B(\mc{G}),\chi_B(\mc{G}')\}$?    \end{question}
	
	A fundamental tool for the investigation of Hedetniemi's conjecture is the $n$-coloring graph $C_n(G)$ of a graph $G$ defined by El-Zahar and Sauer \cite{zahar}. It is not clear, however, whether there exist analogous well-behaving objects for Borel graphs.
	
	\begin{problem}
		Let $\mc{G}$ be a Borel graph. Define a graph $C_n(\mc{G})$ of $n$-colorings of $\mc{G}$ for which the results of El-Zahar and Sauer \cite{zahar} can be generalized.
	\end{problem}

	One could hope for a positive result after excluding the sort of examples constructed in this paper. More precisely, our example can be viewed as follows: a smooth equivalence relation $E$ has been constructed so that there are no $E_\mc{S}$ edges between the classes (in other words $E$ is a smooth super-equivalence relation of a restriction of $E_0$ to some Borel set) and each $E$ class has finite Borel chromatic number, but the union of $E$ classes has infinite Borel chromatic number. Note also that such a graph still has a $\mb{\Sigma}^1_2$-measurable finite coloring. Hence, the following questions are natural:
	
	\begin{question} Let $B \subset \omm$ be an $E_0$-invariant Borel set (that is, it is the union of $E_\mc{S}$ connected components).
		\begin{enumerate}
			\item Suppose that there is no smooth super-equivalence relation $E$ of $E_0|_B$ so that for every $x \in B$ we have $\chi_B(\mc{G}_\mc{S}|_{[x]_E})<\aleph_0$. Does $\mc{G}_\mc{S} \leq_B \mc{G}_\mc{S}|_B$ hold?
			\item (PD) Can we formulate basis results for graphs without finite ``definable" colorings? For instance, suppose that the graph $\mc{G}_\mc{S}|_B$ has no projective finite coloring. Does $\mc{G}_\mc{S} \leq_B \mc{G}_\mc{S}|_B$ hold?
		\end{enumerate}
	\end{question}
	
	Finally, from an affirmative answer to the following question one could give a different proof of Theorem \ref{t:main1}, inferring it directly from Corollary \ref{c:kechadams}.
	
	\begin{question}
		\label{q:unif} Let $B \subset \oom \times [\om]^\om$ be a Borel set such that for all $x \in \oom$ the graph $\mc{G}_\mc{S}|_{B_x}$ has finite Borel chromatic number. Does $B^c$ have a full Borel uniformization?
		
	\end{question}

	\bibliographystyle{apalike}
	\bibliography{anti}

\begin{thebibliography}{}

\bibitem[Adams and Kechris, 2000]{adams2000linear}
Adams, S. and Kechris, A.~S. (2000).
\newblock Linear algebraic groups and countable {B}orel equivalence relations.
\newblock {\em J. Amer. Math. Soc.}, 13(4):909--943.

\bibitem[Carroy et~al., 2021]{benen}
Carroy, R., Miller, B.~D., Schrittesser, D., and Vidnyanszky, Z. (2021).
\newblock Minimal definable graphs of definable chromatic number at least
  three.
\newblock {\em Forum of Math. Sigma}, 9(e7).

\bibitem[Chong and Yu, 2015]{chong2015recursion}
Chong, C.~T. and Yu, L. (2015).
\newblock {\em Recursion theory}, volume~8 of {\em De Gruyter Series in Logic
  and its Applications}.
\newblock De Gruyter, Berlin.
\newblock Computational aspects of definability, With an interview with Gerald
  E. Sacks.

\bibitem[Conley and Kechris, 2013]{conleymeas}
Conley, C.~T. and Kechris, A.~S. (2013).
\newblock Measurable chromatic and independence numbers for ergodic graphs and
  group actions.
\newblock {\em Groups Geom. Dyn.}, 7(1):127--180.

\bibitem[Conley and Miller, 2014]{conley2014antibasis}
Conley, C.~T. and Miller, B.~D. (2014).
\newblock An antibasis result for graphs of infinite {B}orel chromatic number.
\newblock {\em Proc. Amer. Math. Soc.}, 142(6):2123--2133.

\bibitem[Di~Prisco and Todorcevic, 2015]{di2015basis}
Di~Prisco, C.~A. and Todorcevic, S. (2015).
\newblock Basis problems for {B}orel graphs.
\newblock {\em Zb. Rad.(Beogr.), Selected topics in combinatorial analysis},
  17(25):33--51.

\bibitem[El-Zahar and Sauer, 1985]{zahar}
El-Zahar, M. and Sauer, N.~W. (1985).
\newblock The chromatic number of the product of two {$4$}-chromatic graphs is
  {$4$}.
\newblock {\em Combinatorica}, 5(2):121--126.

\bibitem[Hajnal, 1985]{hajnal}
Hajnal, A. (1985).
\newblock The chromatic number of the product of two {$\aleph_1$}-chromatic
  graphs can be countable.
\newblock {\em Combinatorica}, 5(2):137--139.

\bibitem[Hjorth, 20]{hjorth}
Hjorth, G. (20??).
\newblock Complexity of non-dominating sets.
\newblock {\em Unpublished notes}.

\bibitem[Kechris, 1995]{kechrisbook}
Kechris, A.~S. (1995).
\newblock {\em Classical descriptive set theory}, volume 156 of {\em Graduate
  Texts in Mathematics}.
\newblock Springer-Verlag, New York.

\bibitem[Kechris and Marks, 2015]{kechris2015descriptive}
Kechris, A.~S. and Marks, A.~S. (2015).
\newblock Descriptive graph combinatorics.
\newblock {\em Preprint available at http://math. ucla. edu/\~{} marks}.

\bibitem[Kechris et~al., 1999]{kechris1999borel}
Kechris, A.~S., Solecki, S., and Todorcevic, S. (1999).
\newblock Borel chromatic numbers.
\newblock {\em Adv. Math.}, 141(1):1--44.

\bibitem[Mansfield and Weitkamp, 1985]{mansfield}
Mansfield, R. and Weitkamp, G. (1985).
\newblock {\em Recursive aspects of descriptive set theory}, volume~11 of {\em
  Oxford Logic Guides}.
\newblock The Clarendon Press, Oxford University Press, New York.
\newblock With a chapter by Stephen Simpson.

\bibitem[Marcone, 1995]{marcone}
Marcone, A. (1995).
\newblock The set of better quasi orderings is {${\bf \Pi}_2^1$}.
\newblock {\em Math. Logic Quart.}, 41(3):373--383.

\bibitem[Marks, 2016]{marksdet}
Marks, A.~S. (2016).
\newblock A determinacy approach to {B}orel combinatorics.
\newblock {\em J. Amer. Math. Soc.}, 29(2):579--600.

\bibitem[Miller, 2012]{miller2012graph}
Miller, B.~D. (2012).
\newblock The graph-theoretic approach to descriptive set theory.
\newblock {\em Bull. Symbolic Logic}, 18(4):554--575.

\bibitem[Moschovakis, 2009]{moschovakis2009descriptive}
Moschovakis, Y.~N. (2009).
\newblock {\em Descriptive set theory}, volume 155 of {\em Mathematical Surveys
  and Monographs}.
\newblock American Mathematical Society, Providence, RI, second edition.

\bibitem[Nash-Williams, 1967]{nashwilliams}
Nash-Williams, C. S. J.~A. (1967).
\newblock On well-quasi-ordering trees.
\newblock In {\em A seminar on {G}raph {T}heory}, pages 79--82. Holt, Rinehart
  and Winston, New York.

\bibitem[Pequignot, 2017a]{pequignot2017finite}
Pequignot, Y. (2017a).
\newblock Finite versus infinite: an insufficient shift.
\newblock {\em Adv. Math.}, 320(1):244--249.

\bibitem[Pequignot, 2017b]{Peq}
Pequignot, Y. (2017b).
\newblock Towards better: a motivated introduction to better-quasi-orders.
\newblock {\em EMS Surv. Math. Sci.}, 4(2):185--218.

\bibitem[Sabok, 2012]{sabok2012complexity}
Sabok, M. (2012).
\newblock Complexity of {R}amsey null sets.
\newblock {\em Adv. Math.}, 230(3):1184--1195.

\bibitem[Sauer, 2001]{sauer2001hedetniemi}
Sauer, N. (2001).
\newblock Hedetniemi's conjecture---a survey.
\newblock {\em Discrete Math.}, 229(1-3):261--292.
\newblock Combinatorics, graph theory, algorithms and applications.

\bibitem[Shitov, 2019]{shitov}
Shitov, Y. (2019).
\newblock Counterexamples to {H}edetniemi's conjecture.
\newblock {\em Ann. of Math. (2)}, 190(2):663--667.

\bibitem[Simpson, 2009]{simpson}
Simpson, S.~G. (2009).
\newblock {\em Subsystems of second order arithmetic}.
\newblock Perspectives in Logic. Cambridge University Press, Cambridge;
  Association for Symbolic Logic, Poughkeepsie, NY, second edition.

\bibitem[Solovay, 1978]{solovay}
Solovay, R.~M. (1978).
\newblock Hyperarithmetically encodable sets.
\newblock {\em Trans. Amer. Math. Soc.}, 239:99--122.

\end{thebibliography}

	\newpage
	
	Stevo Todor\v{c}evi\'c
	
	Institut de Math\'emathiques de Jussieu,
	
	CNRS UMR 7586, Case 247,
	
	4 place Jussieu, 75252 Paris Cedex, France
	
	stevo@math.jussieu.fr
	
	and  
	
	Department of Mathematics, University of Toronto,
	
	Toronto, Canada M5S 3G3
	
	stevo@math.toronto.edu
	
	\bigskip
	
	Zolt\'an Vidny\'anszky
	
	Department of Mathematics, 
	
	California Institute of Technology,
	
    Pasadena, CA 91125
	
	vidnyanz@caltech.edu
	
\end{document}